\documentclass[a4paper] {article}
[12pt]
    
\usepackage[top=3.28cm, bottom=3.28cm, left=2.88cm, right=2.88cm]{geometry}

\makeatletter
\renewcommand*\l@section{\@dottedtocline{1}{1.5em}{2.3em}}
\makeatother

\usepackage{amsfonts}
\usepackage{amssymb}
\usepackage[T1]{fontenc}

\usepackage{tikz}
\usetikzlibrary{calc}

\usepackage{CJK}
\usepackage{amsmath}
 
\usepackage{amsfonts}
\usepackage{amssymb}
\usepackage{amsthm}
\usepackage{amssymb}
\usepackage{enumerate}
\usepackage[calc]{picture}
\usepackage[all,cmtip]{xy}

\usepackage[mathscr]{eucal}
\usepackage{eqlist}

\usepackage{color}
\usepackage{abstract} 
\usepackage[T1]{fontenc}
 
\theoremstyle{plain}
\newtheorem{theorem}{Theorem}
\newtheorem{proposition}[theorem]{Proposition}
\newtheorem{lemma}[theorem]{Lemma}

\newtheorem{corollary}[theorem]{Corollary}

\theoremstyle{definition}
\newtheorem{definition}{Definition}

\usepackage{etoolbox}
\newtheoremstyle{myrem}
 {3pt}
 {3pt}
 {\normalsize}
 { }
 {\itshape}
 {:}
 { }
 {}

 \theoremstyle{myrem}
 \newtheorem{remark}{Remark}
 \appto\remark{\leftskip\parindent}
 \appto\remark{\rightskip\parindent}

\numberwithin{equation}{section}
\numberwithin{theorem}{section}

\begin{document}

\begin{center}
{\Large {\textbf {
 Simplicial-like Identities for  The  Paths  and  The Regular Paths on Discrete Sets
}}}
 \vspace{0.58cm}\\

Shiquan Ren


\bigskip

\begin{quote}
\begin{abstract}
Simplicial identities play an important  and fundamental role   in  simplicial homotopy theory.  On the other hand,  the study of the paths  and the regular paths  on discrete sets   is  the foundation for the path-homology theory of digraphs. In this paper,  by investigating some weighted face maps and weighted co-face maps on the space of the paths as well as the space of the regular paths,  we prove some simplicial-like identities for  the   paths  and the regular paths on discrete sets. 
\end{abstract}

\bigskip



\end{quote}

\end{center}

\bigskip

\section{Introduction}

Simplicial identities are important tools  in  simplicial homotopy theory.   So far,  topologists applied   simplicial   methods in algebraic topology  systematically  and developed the simplicial homotopy theory.     For instances,     Edward  B. Curtis  \cite{curtis}  in 1971,    Goerss, Paul G. and  Jardine, John \cite{go}  in 2009,      Jie Wu \cite{jiewu2}  in 2010,  etc.  Moreover,  essential applications of simplicial homotopy theory  have been found in various areas.  For example,  F. R. Cohen  and J. Wu \cite{cohen}  in braids,   Fengchun  Lei,  Fengling Li,  Jie Wu \cite{jiewu1}  in framed links,   Fedor Pavutnitskiy  and Jie Wu \cite{wu5}  in Adams spectral sequence,  etc.

In 1990s,   A. Dimakis  and   F. M\"{u}ller-Hoissen  \cite{d1,d2,d3}   initiated the study of  discrete differential calculus   and discrete Riemannian geometry  with  a motivation in  theoretical physics.  During 2010s,   based on the study of A. Dimakis  and   F. M\"{u}ller-Hoissen  \cite{d1,d2,d3},   Alexander Grigor'yan,  Yong  Lin  and   Shing-Tung Yau  \cite{lin1},   Alexander Grigor'yan,  Yong  Lin, Yuri Muranov and  Shing-Tung Yau \cite{lin2,lin3,lin4}  and   Alexander Grigor'yan,   Yuri Muranov and  Shing-Tung Yau \cite{lin5,lin6}  developed a path-homology theory of  digraphs.   The  theoretical foundation for the path-homology in \cite{lin1} - \cite{lin6}  relies on  the theory of the  paths  and the  regular paths  on discrete sets.

In this paper,  we study some weighted face maps and some  weighted co-face maps for the paths and the regular paths on discrete sets.  We prove some simplicial-like identities for  the   paths  and the regular paths on discrete sets.  The main results of this paper are Theorem~\ref{pr-2.3.2}  and Theorem~\ref{pr-3.1.1}.

  As by-products,    we derive some weighted boundary operators from the weighted face maps and derive some weighted co-boundary operators from the weighted co-face maps.  We  prove the anti-commutative rule and the Newton-Leibniz rule for the weighted boundary operators and the weighted co-boundary operators of the  paths on   discrete sets, in Subsection~\ref{ss-2.3};  we calculate the anti-commutators   for the  weighted boundary operators and the weighted co-boundary operators of  the  regular paths on   discrete sets,  in Subsection~\ref{ss-3.2}. 

\smallskip

\section{Simplicial-Like Identities for Paths on Discrete Sets}

Let  $V$  be  a   finite set whose elements are called {\it vertices}.  Let  $n\geq 0$  be a non-negative  integer.   An {\it elementary $n$-path}  on $V$  is an ordered  sequence   $v_0v_1\ldots v_n$  where $v_0,v_1,\ldots, v_n$ are  vertices in $V$  (cf. \cite[Definition~2.1]{lin2}).  Here  for any $0\leq i<j\leq n$,  the vertices $v_i$  and $v_j$ are not assumed to be distinct.   Let $\mathbb{R}$  be the real numbers.  A linear combination 
\begin{eqnarray*}
\sum_{v_0,v_1,\ldots, v_n\in V}  r_{v_0v_1\ldots v_n}  v_0v_1\ldots v_n,~~~r_{v_0v_1\ldots v_n} \in\mathbb{R},   
\end{eqnarray*}
of elementary $n$-paths on $V$  is called an {\it $n$-path}  on $V$  (cf. \cite[Definition~2.2]{lin2}).  Let  $\Lambda_n(V)$  be the vector space consisting of all the $n$-paths on $V$ (cf. \cite[Definition~2.2]{lin2}).  Note that for  any $n\geq 0$,  the space $\Lambda_n(V)$  is  of   dimension $(\#V)^{n+1}$.   In particular, $\Lambda_0(V)$  is the vector space spanned by all the vertices in $V$   thus is of dimension $\# V$.      
We consider the graded vector space
\begin{eqnarray*}
\Lambda_*(V)=\bigoplus_{n=0}^\infty\Lambda_n(V)
\end{eqnarray*}
with the canonical  addition $+$ and the  canonical (real) scalar multiplication.  
Let $n$ and  $m$  be two non-negative integers.  We take an    $n$-path 
\begin{eqnarray}\label{eq-2.0.1}
\xi=\sum_{v_0,v_1,\ldots,v_n\in V} r_{v_0v_1\ldots v_n}  v_0v_1\ldots v_n
\end{eqnarray}
in $\Lambda_n(V)$ and take an $m$-path  
\begin{eqnarray}\label{eq-2.0.2}
\eta=\sum_{u_0,u_1,\ldots,u_m\in V} t_{u_0u_1\ldots u_n}  u_0u_1\ldots u_m 
\end{eqnarray}
in $\Lambda_m(V)$.    We define their {\it join}   $\xi * \eta$   to be an $(n+m+1)$-path in $\Lambda_{n+m+1}(V)$  by letting  (cf. \cite[Subsection~2.2, Join of paths]{lin2})
\begin{eqnarray}\label{eq-2.0.3}
\xi*\eta=\sum_{v_0,v_1,\ldots,v_n\in V;\atop u_0,u_1,\ldots, u_m\in V}  r_{v_0v_1\ldots v_n}  t_{u_0u_1\ldots u_m} v_0v_1\ldots v_n u_0u_1\ldots u_m. 
\end{eqnarray}
We extend $*$  blinearly over $\Lambda_*(V)$.  
It  is direct to verify the following laws: 
\begin{enumerate}[(i).]
\item
bilinear law:  $(\lambda\xi)*(\mu\eta)=(\lambda\mu) (\xi*\eta)$ for any $\lambda,\mu\in\mathbb{R}$  and   any $\xi,\eta\in \Lambda_*(V)$; 
\item 
associative law:     $(\xi*\eta)*\theta=\xi* (\eta*\theta)$ for   any $\xi, \eta,\theta\in \Lambda_*(V)$; 
\item
 distributive law:  $\xi*(\eta+\theta)=\xi*\eta+ \xi* \theta$  and  $(\eta+\theta)*\xi=\eta* \xi+ \theta*\xi$   for any  $\xi, \eta,\theta\in \Lambda_*(V)$. 
\end{enumerate}
We  give the next definition.    
\begin{definition}
We call the graded vector space $\Lambda_*(V)$ equipped with the  join  $*$ the {\it path algebra}  on $V$ with coefficients in the real numbers. 
\end{definition}

\smallskip


\subsection{The Partial Derivatives and the Partial Differentiations  for Paths}\label{ss2.2}

Let  $v\in V$.  In this subsection,  we define the partial derivative   as well as  the partial differentiation on $\Lambda_*(V)$, with respect to $v$.  

\smallskip

\begin{definition}\label{def-2.222}
 For each $n\geq 0$,  we    define the {\it partial derivative with respect to $v$} to be   a linear map 
\begin{eqnarray*}
\vec{\frac{\partial}{\partial v}}:~~~ \Lambda_n(V)\longrightarrow  \prod_{n+1}\Lambda_{n-1}(V)
\end{eqnarray*}
by 
\begin{eqnarray}\label{eq-apr.8}
\vec{\frac{\partial}{\partial v}}=\Big(\frac{\partial_0}{\partial v},   \frac{\partial_1}{\partial v},\ldots, \frac{\partial_n}{\partial v}\Big). 
\end{eqnarray}
Here in (\ref{eq-apr.8}),  for each  $0\leq i\leq n$,  the $i$-th coordinate of (\ref{eq-apr.8})  is a linear map
\begin{eqnarray*}
 \frac{\partial_i}{\partial v}:~~\Lambda_n(V)\longrightarrow   \Lambda_{n-1}(V)  
 \end{eqnarray*}
given by 
\begin{eqnarray}\label{eq-2.4ax}
\frac{\partial_i}{\partial v}(v_0v_1\ldots v_n)=
(-1)^i\delta(v,v_i)  v_0\ldots\widehat{v_i} \ldots v_n    
\end{eqnarray} 
for any elementary $n$-path $v_0v_1\ldots v_n$ on $V$.    
\end{definition}

\begin{remark}
We give a remark on the notation in (\ref{eq-2.4ax}).  For any $v,u\in V$,  the notation $\delta(v,u)$  is defined to be $1$  if $v=u$  and is  defined to be $0$  if  $v\neq u$.  
\end{remark}


\begin{definition}\label{def-2.223}
For each  $n\geq 0$,  we     define the {\it partial differentiation with respect to $v$} to be  a linear map 
\begin{eqnarray*}
\vec{d}v:~~~  \Lambda_n(V)\longrightarrow\prod_{n+2} \Lambda_{n+1}(V)
\end{eqnarray*}
by 
\begin{eqnarray}\label{eq-apr.9}
\vec{d}v=(d_0 v,  d_1 v,\ldots, d_{n+1} v ). 
\end{eqnarray}  
Here in (\ref{eq-apr.9}),  for each  $0\leq i\leq n+1$,  the $i$-th coordinate of (\ref{eq-apr.9})  is a linear map
\begin{eqnarray*}
d_i v:~~\Lambda_n(V)\longrightarrow   \Lambda_{n+1}(V)  
 \end{eqnarray*}
given by 
\begin{eqnarray}\label{eq-2.2.1}
d_i v(v_0v_1\ldots v_n)=
(-1)^i v_0\ldots v_{i-1} v v_i\ldots v_n
\end{eqnarray} 
 for each elementary $n$-path $v_0v_1\ldots v_n$  on $V$.  
 \end{definition}
 
 \begin{remark} 
 We  give a remark for Definition~\ref{def-2.223}.  
 Letting $i=0$  in (\ref{eq-2.2.1}),  we have 
 \begin{eqnarray}\label{eq-2.2.2}
d_0 v(v_0v_1\ldots v_n)=
(-1)^0 v v_0\ldots v_n; 
\end{eqnarray}   
  letting $i=n+1$  in (\ref{eq-2.2.1}),  we have 
 \begin{eqnarray}\label{eq-2.2.3}
d_0 v(v_0v_1\ldots v_n)=
(-1)^{n+1}v_0\ldots v_n v.  
\end{eqnarray}   
 Hence for each elementary $n$-path $v_0v_1\ldots v_n$  on $V$,  by (\ref{eq-2.2.1}), (\ref{eq-2.2.2})  and (\ref{eq-2.2.3}),   its  image under $\vec{d}v$  is a vector 
\begin{eqnarray*}
\vec{d}v(v_0v_1\ldots v_n)&=&\Big((-1)^0 v v_0\ldots v_n,  (-1)^1 v_0 v v_1\ldots v_n,  
\ldots, \\
&&(-1)^i v_0\ldots v_{i-1} v v_i\ldots v_n,\ldots, (-1)^{n+1}v_0\ldots v_n v\Big). 
\end{eqnarray*} 
\end{remark}

By Definition~\ref{def-2.222}  and  Definition~\ref{def-2.222},   we  have the next three lemmas.  

\begin{lemma}\label{le-2.2.1}
Let $v,u\in V$. Let $n\geq 0$.  Then for any possible $i$ and $j$ with $i<j$, we have  
\begin{eqnarray*}
\frac{\partial_i}{\partial u}\circ \frac{\partial_j}{\partial v}=-\frac{\partial_{j-1}}{\partial v}\circ\frac{\partial_i}{\partial u}.
\end{eqnarray*} 
\end{lemma}

\begin{proof}
Let  $v,u\in V$ and $n\geq 0$.  For any $0\leq j\leq n$  and  any $0\leq i\leq n-1$,  it follows from a direct calculation that  
\begin{eqnarray*}
\frac{\partial_i}{\partial u}\circ \frac{\partial_j}{\partial v}(v_0v_1\ldots v_n)
&=&\frac{\partial_i}{\partial u}\Big((-1)^j\delta(v,v_j)  v_0\ldots\widehat{v_j} \ldots v_n  
\Big)\\
 &=&(-1)^j\delta(v,v_j) \frac{\partial_i}{\partial u} (v_0\ldots\widehat{v_j} \ldots v_n)\\
 &=&  
 \begin{cases}
 (-1)^{i+j}\delta(v,v_j)\delta(u,v_i)  v_0\ldots\widehat{v_i}\ldots \widehat{v_j} \ldots v_n,  & 0\leq i\leq j-1;\\
 (-1)^{i+j}\delta(v,v_j)\delta(u,v_{i+1})  v_0\ldots\widehat{v_j}\ldots \widehat{v_{i+1}} \ldots v_n, &j\leq i\leq n-1.
  \end{cases}
 \end{eqnarray*}
Exchanging $u$ and $v$,  we have 
 \begin{eqnarray*}
\frac{\partial_i}{\partial v}\circ \frac{\partial_j}{\partial u}(v_0v_1\ldots v_n)
 &=&  
 \begin{cases}
 (-1)^{i+j}\delta(u,v_j)\delta(v,v_i)  v_0\ldots\widehat{v_i}\ldots \widehat{v_j} \ldots v_n,  & 0\leq i\leq j-1;\\
 (-1)^{i+j}\delta(u,v_j)\delta(v,v_{i+1})  v_0\ldots\widehat{v_j}\ldots \widehat{v_{i+1}} \ldots v_n, &j\leq i\leq n-1. 
  \end{cases}
 \end{eqnarray*}
Now we suppose $i<j$.  It follows from the above  two equations that for any elementary $n$-path $v_0v_1\ldots v_n$ on $V$,  we  have 
\begin{eqnarray*}
\frac{\partial_i}{\partial u}\circ \frac{\partial_j}{\partial v}(v_0v_1\ldots v_n)=-\frac{\partial_{j-1}}{\partial v}\circ\frac{\partial_i}{\partial u}(v_0v_1\ldots v_n). 
\end{eqnarray*}
By the linear property  of $\frac{\partial_i}{\partial u}\circ \frac{\partial_j}{\partial v}$  and  $\frac{\partial_{j-1}}{\partial v}\circ\frac{\partial_i}{\partial u}$,  the lemma follows.
\end{proof}

\begin{lemma}\label{le-2.2.2}
Let $v,u\in V$. Let $n\geq 0$.  Then  for any possible $i$ and $j$  we have  
\begin{eqnarray*}
\frac{\partial_i}{\partial u}\circ d_j v = \begin{cases}
- d_{j-1} v \circ\dfrac{\partial_i}{\partial u},  & i<j;\\
\delta(u,v)  {\rm~ id},   & i=j; \\
- d_j v \circ \dfrac{\partial_{i-1}}{\partial u},  &  i>j. 
\end{cases}
\end{eqnarray*}
\end{lemma}
\begin{proof}
Let  $v,u\in V$ and $n\geq 0$.  Let $0\leq i, j\leq n+1$. Take an elementary $n$-path $v_0v_1\ldots v_n$ on $V$.  We consider   three cases:

{\sc Case~1}.  $i<j$.  Then 
\begin{eqnarray*}
\frac{\partial_i}{\partial u}\circ d_j v(v_0v_1\ldots v_n)&=& (-1)^{i+j} \delta( u,v_i) v_0\ldots \widehat{v_i} \ldots v_{j-1} v v_j \ldots v_n\\
&=& -(-1)^{i+(j-1)} \delta( u,v_i) v_0\ldots \widehat{v_i} \ldots v_{j-1} v v_j \ldots v_n\\
&=&- d_{j-1} v \circ\dfrac{\partial_i}{\partial u}(v_0v_1\ldots v_n).  
\end{eqnarray*}
Thus by the linear property  of $\frac{\partial_i}{\partial u}\circ d_j v $  and $d_{j-1} v \circ\frac{\partial_i}{\partial u}$,  we have 
\begin{eqnarray*}
\frac{\partial_i}{\partial u}\circ d_j v=- d_{j-1} v \circ\dfrac{\partial_i}{\partial u}.  
\end{eqnarray*}

{\sc Case~2}.  $i=j$.  Then 
\begin{eqnarray*}
\frac{\partial_i}{\partial u}\circ d_i v(v_0v_1\ldots v_n)&=& (-1)^{i+i} \delta( u,v) v_0\ldots  v_{i-1} \widehat{v} v_i \ldots\ldots v_n\\
&=& \delta( u,v) v_0v_1\ldots v_n.  
\end{eqnarray*}
Thus by the linear property of $\frac{\partial_i}{\partial u}\circ d_j v$,  we have 
\begin{eqnarray*}
\frac{\partial_i}{\partial u}\circ d_j v=\delta(u,v)  {\rm~ id}.  
\end{eqnarray*}

{\sc Case~3}.  $i>j$.  Then 
\begin{eqnarray*}
\frac{\partial_i}{\partial u}\circ d_j v(v_0v_1\ldots v_n)&=& (-1)^{i+j} \delta( u,v_i) v_0\ldots  v_{j-1} v v_j \ldots\widehat{v_{i-1}} \ldots v_n\\
&=& -(-1)^{(i-1)+j} \delta( u,v_i) v_0\ldots v_{j-1} v v_j \ldots \widehat{v_{i-1}} \ldots v_n\\
&=&- d_{j} v \circ\dfrac{\partial_{i-1}}{\partial u}(v_0v_1\ldots v_n).  
\end{eqnarray*}
Thus by the linear property  of $\frac{\partial_i}{\partial u}\circ d_j v$ and  $d_{j} v \circ\frac{\partial_{i-1}}{\partial u}$,  we have 
\begin{eqnarray*}
\frac{\partial_i}{\partial u}\circ d_j v=- d_{j} v \circ\dfrac{\partial_{i-1}}{\partial u}.  
\end{eqnarray*}

Summarizing all the three cases,  we obtain the lemma.  
\end{proof}

\begin{lemma}\label{le-2.2.3}
Let $v,u\in V$. Let $n\geq 0$.  Then  for any possible $i$ and $j$ with $i\leq  j$,  we have  
\begin{eqnarray*}
d_i u \circ d_j v= -d_{j+1}  v \circ  d_i u.  
\end{eqnarray*}
\end{lemma}

\begin{proof}
Let  $v,u\in V$ and $n\geq 0$. Take an elementary $n$-path $v_0v_1\ldots v_n$ on $V$.  Then for any possible $i$ and $j$ we have 
\begin{eqnarray*}
d_i u \circ d_j v (v_0v_1\ldots v_n)&=& (-1)^j d_i u(v_0\ldots v_{j-1} v v_j\ldots v_n)\\
&=&\begin{cases}
(-1)^{i+j}v_0\ldots v_{i-1} u v_i\ldots  v_{j-1} v v_j\ldots v_n,  &i<j;\\
 v_0\ldots v_{i-1} u   v v_i\ldots v_n,&i=j;\\
  -v_0\ldots v_{i-1} v u v_i\ldots v_n,&i=j+1;\\
(-1)^{i+j}v_0\ldots v_{j-1} v v_j\ldots  v_{i-2} u v_{i-1}\ldots v_n, & i>j+1.
\end{cases}
\end{eqnarray*}
 Exchanging $u$ and $v$  we have 
 \begin{eqnarray*}
d_i v \circ d_j u (v_0v_1\ldots v_n)=\begin{cases}
(-1)^{i+j}v_0\ldots v_{i-1} v v_i\ldots  v_{j-1} u v_j\ldots v_n,  &i<j;\\
 v_0\ldots v_{i-1} v   u  v_i\ldots v_n,&i=j;\\
  -v_0\ldots v_{i-1} u v v_i\ldots v_n,&i=j+1;\\
(-1)^{i+j}v_0\ldots v_{j-1} u v_j\ldots  v_{i-2} v v_{i-1}\ldots v_n, & i>j+1.
\end{cases} 
 \end{eqnarray*}
 Now we suppose $i\leq j$.  Then it follows from the above  two equations  that 
 \begin{eqnarray*}
d_i u \circ d_j v (v_0v_1\ldots v_n)=-d_{j+1}  v \circ  d_i u (v_0v_1\ldots v_n).   
 \end{eqnarray*}
 By the linear property  of $d_i u \circ d_j v$  and $d_{j+1}  v \circ  d_i u$,  the lemma follows.  
\end{proof}

By the end of this subsection,   we summarize Lemma~\ref{le-2.2.1},  Lemma~\ref{le-2.2.2} and Lemma~\ref{le-2.2.3}  in the following list:

 \begin{quote}
 \begin{itemize}
 \item
  $\dfrac{\partial_i}{\partial u}\circ \dfrac{\partial_j}{\partial v}=-\dfrac{\partial_{j-1}}{\partial v}\circ\dfrac{\partial_i}{\partial u}$ for $i<j$;
  \item
$
\dfrac{\partial_i}{\partial u}\circ d_j v = \begin{cases}
- d_{j-1} v \circ\dfrac{\partial_i}{\partial u},  & i<j,\\
\delta(u,v)  {\rm~ id},   & i=j, \\
- d_j v \circ \dfrac{\partial_{i-1}}{\partial u},  &  i>j; 
\end{cases}
$
\item
$d_i u \circ d_j v= -d_{j+1}  v \circ  d_i u$  for $i\leq j$.    
\end{itemize}
 \end{quote}
 
\smallskip

\subsection{The Weighted Face Maps and the Weighted Co-Face Maps  for Paths}\label{ss-2.22}

Let $f:  V\longrightarrow \mathbb{R}$  be an real valued function on $V$ which assigns a real number $f(v)$ to each vertex $v\in V$.   Let $n\geq 0$.  In this subsection,  we define the $f$-weighted face maps as well as  the  $f$-weighted co-face maps  on $\Lambda_*(V)$.  Then we prove some simplicial-like identities. 

\smallskip

\begin{definition}\label{def-2.30}
  For each $0\leq i\leq n$,  we define the {\it $i$-th $f$-weighted face map} to be a linear map 
\begin{eqnarray*}
\sum_{v\in V}  f(v) \frac{\partial_i}{\partial v}:~~~\Lambda_n(V)\longrightarrow \Lambda_{n-1}(V). 
\end{eqnarray*}
For simplicity,  we use the notation 
\begin{eqnarray}\label{eq-2.3.8}
\partial_i^f:=\sum_{v\in V}  f(v) \dfrac{\partial_i}{\partial v}
\end{eqnarray}
for the $i$-th $f$-weighted face map.  
\end{definition}

\begin{definition}\label{def-2.31}
For each $0\leq i\leq n+1$,  we define the {\it $i$-th $f$-weighted co-face map} to be a linear map 
\begin{eqnarray*}
\sum_{v\in V}  f(v)  d_i v:~~~\Lambda_n(V)\longrightarrow \Lambda_{n+1}(V). 
\end{eqnarray*}
For simplicity,  we use the notation 
 \begin{eqnarray}\label{eq-2.3.9}
 d_i^f:=\sum_{v\in V}  f(v) d_i  v
 \end{eqnarray}
 for the $i$-th  $f$-weighted co-face map. 
\end{definition}

 \begin{definition} \label{def-2.32}
 For any real functions $f$ and $g$ on $V$,  their inner product with respect to $V$ is defined by 
 \begin{eqnarray*}
 \langle f,g\rangle =\langle f,g\rangle^V:= \sum_{v\in V}  f(v) g(v).  
 \end{eqnarray*}
 \end{definition}
 
 \begin{definition}
Taking $f=g$ in  Definition~\ref{def-2.32},  for any real function $f$  on $V$,   the  $L^2$-norm of $f$ is defined by 
 \begin{eqnarray*}
 ||f||_2=||f||_2^V:=\sqrt{\langle f,f\rangle }=\sqrt {\sum_{v\in V}  |f(v)|^2}.  
 \end{eqnarray*}  
 \end{definition}

The next proposition gives the explicit formulas for $\partial_i^f$ and $d_i^f$  on  the  elementary $n$-paths. 

\begin{proposition}\label{le-2.3.99}
Let $f$  be a real function  on $V$.  Let  $n\geq 0$ and let $v_0v_1\ldots v_n$  be an  elementary $n$-path   on $V$.  Then for any   $0\leq i\leq n$  we  have 
\begin{eqnarray*}
\partial^f_i (v_0v_1\ldots v_n)=(-1)^i f(v_i) v_0\ldots \widehat{v_i} \ldots v
\end{eqnarray*}
and for any $0\leq j\leq n+1$ we have 
\begin{eqnarray*}
d^f_j (v_0v_1\ldots v_n)=\sum_{v\in V} (-1)^j f(v) v_0\ldots  v_{j-1} v v_j\ldots v_n. 
\end{eqnarray*}
\end{proposition} 
 \begin{proof}
 Let $0\leq i\leq n$ and $0\leq j\leq n+1$.  By a straight-forward calculation,  
  \begin{eqnarray*}
   \partial^f_i (v_0v_1\ldots v_n)&=&\sum_{v\in V} f(v)\frac{\partial_i}{\partial v}(v_0v_1\ldots v_n)\\
  &=&\sum_{v\in V}(-1)^i\delta(v,v_i) f(v) v_0\ldots \widehat{v_i} \ldots v\\
  &=& (-1)^i f(v_i) v_0\ldots \widehat{v_i} \ldots v 
  \end{eqnarray*} 
  and 
  \begin{eqnarray*}
  d^f_j(v_0v_1\ldots v_n)&=&  \sum_{v\in V} f(v) d_jv (v_0v_1\ldots v_n)\\
  &=& \sum_{v\in V} (-1)^j f(v) v_0\ldots  v_{j-1} v v_j\ldots v_n. 
  \end{eqnarray*}
  The proposition follows.  
 \end{proof}
 
The next theorem follows with the help of  Subsection~\ref{ss2.2}.

\begin{theorem}[Main Result I:  The simplicial-like identities for paths on discrete sets]
\label{pr-2.3.2}
Let  $f$  and $g$ be  two real functions on $V$.   Then for any $n\geq 0$,  we  have 
\begin{enumerate}[(i).]
\item
$\partial_i^f\circ \partial_j^g=-\partial_{j-1}^g\circ \partial_i^f$ for any $i<j$; 

\item
$
\partial_i ^f\circ d_j^g= \begin{cases}
- d_{j-1}^g \circ\partial_i^f,  & i<j,\\
\langle f,g\rangle  {\rm~ id},   & i=j, \\
- d_j^g \circ \partial_{i-1}^f,  &  i>j; 
\end{cases}
$

\item
$d_i^f \circ  d_j^g= - d_{j+1}^g \circ  d_i^f$  for $i\leq j$.    
\end{enumerate}
\end{theorem}

\begin{proof}
(i).  Suppose $i<j$.  By a straight-forward calculation and   Lemma~\ref{le-2.2.1},    we  have 
\begin{eqnarray*}
\Big(\sum_{v\in V}  f(v) \dfrac{\partial_i}{\partial v}\Big)\circ \Big(\sum_{v\in V}  g(v) \dfrac{\partial_j}{\partial v}\Big)
&=&\Big(\sum_{u\in V}  f(u) \dfrac{\partial_i}{\partial u}\Big)\circ \Big(\sum_{v\in V}  g(v) \dfrac{\partial_j}{\partial v}\Big) \\
&=&\sum_{u,v\in V} f(u) g(v) \dfrac{\partial_i}{\partial u}\circ \dfrac{\partial_j}{\partial v}\\
&=&-\sum_{u,v\in V} f(u) g(v)\dfrac{\partial_{j-1}}{\partial v}\circ\dfrac{\partial_i}{\partial u}\\
&=&-\Big(\sum_{v\in V}  g(v) \dfrac{\partial_{j-1}}{\partial v}\Big)\circ \Big(\sum_{u\in V}  f(u) \dfrac{\partial_i}{\partial u}\Big)\\
&=&-\Big(\sum_{v\in V}  g(v) \dfrac{\partial_{j-1}}{\partial v}\Big)\circ \Big(\sum_{v\in V}  f(v) \dfrac{\partial_i}{\partial v}\Big). 
\end{eqnarray*}
Thus  we obtain
\begin{eqnarray*}
\Big(\sum_{v\in V}  f(v) \dfrac{\partial_i}{\partial v}\Big)\circ \Big(\sum_{v\in V}  g(v) \dfrac{\partial_j}{\partial v}\Big)=-\Big(\sum_{v\in V}  g(v) \dfrac{\partial_{j-1}}{\partial v}\Big)\circ \Big(\sum_{v\in V}  f(v) \dfrac{\partial_i}{\partial v}\Big).  
\end{eqnarray*}
By  the notations  (\ref{eq-2.3.8})  and (\ref{eq-2.3.9}),  we obtain (i).

(ii). By a straight-forward calculation  Lemma~\ref{le-2.2.2},  we have  
\begin{eqnarray*}
\Big(\sum_{v\in V} f(v)\dfrac{\partial_i}{\partial v}\Big)\circ\Big(\sum_{v\in V} g(v) d_j v\Big) &=&\Big(\sum_{u\in V} f(u)\dfrac{\partial_i}{\partial u}\Big)\circ\Big(\sum_{v\in V} g(v) d_j v\Big) \\
&=&\sum_{u,v\in V}  f(u) g(v) \dfrac{\partial_i}{\partial u}\circ d_j v\\
&=&\begin{cases}
-\sum_{u,v\in V} f(u) g(v)  d_{j-1} v \circ\dfrac{\partial_i}{\partial u},  & i<j,\\
\sum_{u,v\in V} f(u) g(v) \delta(u,v)  {\rm~ id},   & i=j, \\
- \sum_{u,v\in V} f(u) g(v)  d_j v \circ \dfrac{\partial_{i-1}}{\partial u},  &  i>j.  
\end{cases}
\end{eqnarray*}
By a similar calculation in (i),  we  have 
\begin{eqnarray*}
\Big(\sum_{v\in V}f(v) d_{j-1} v \Big)\circ\Big(\sum_{v\in V} g(v) \dfrac{\partial_i}{\partial v}\Big)
&=&\sum_{u,v\in V} f(u) g(v)  d_{j-1} v \circ\dfrac{\partial_i}{\partial u},\\
 \langle f,g\rangle{\rm~id}~=~\Big(\sum_{v\in V} f(v) g(v)\Big)  {\rm~ id}
 &=&\sum_{u,v\in V} f(u) g(v) \delta(u,v)  {\rm~ id}, \\
\Big(\sum_{v\in V} f(v) d_j v\Big) \circ \Big(\sum_{v\in V}  g(v) \dfrac{\partial_{i-1}}{\partial v}\Big)
&=&\sum_{u,v\in V} f(u) g(v)  d_j v \circ \dfrac{\partial_{i-1}}{\partial u}.  
\end{eqnarray*}
Therefore,  it follows from the above  four equations  that 
\begin{eqnarray*}
\Big(\sum_{v\in V} f(v)\dfrac{\partial_i}{\partial v}\Big)\circ\Big(\sum_{v\in V} g(v) d_j v\Big) = \begin{cases}
- \Big(\sum_{v\in V}  f(v)  d_{j-1} v \Big)\circ\Big(\sum_{v\in V}  g(v)  \dfrac{\partial_i}{\partial v}\Big),  & i<j,\\
\Big(\sum_{v\in V} f(v)  g(v)\Big)  {\rm~ id},   & i=j, \\
- \Big(\sum_{v\in V} f(v) d_j v\Big) \circ \Big(\sum_{v\in V}  g(v) \dfrac{\partial_{i-1}}{\partial v}\Big),  &  i>j.  
\end{cases}
\end{eqnarray*}
With the help of  the notations  (\ref{eq-2.3.8})  and (\ref{eq-2.3.9}),  we obtain (ii).

(iii).   Suppose $i\leq j$.  By a straight-forward calculation and   Lemma~\ref{le-2.2.3},    we  have 
\begin{eqnarray*}
\Big(\sum_{v\in V} f(v) d_i v\Big) \circ\Big(\sum_{v\in V}  g(v) d_j v\Big)&=&\sum_{u,v\in V} f(u) g(v)  d_i v \circ d_j u\\
&=&-\sum_{u,v\in V} f(u) g(v) d_{j+1}  v \circ  d_i u\\
&=&-\Big(\sum_{v\in V}  g(v) d_{j+1}  v\Big) \circ  \Big(\sum_{v\in V}  f(v)  d_i v\Big). 
\end{eqnarray*}
Thus we obtain
\begin{eqnarray*}
  \Big(\sum_{v\in V} f(v) d_i v\Big) \circ\Big(\sum_{v\in V}  g(v) d_j v\Big)= -\Big(\sum_{v\in V}  g(v) d_{j+1}  v\Big) \circ  \Big(\sum_{v\in V}  f(v)  d_i v\Big).
  \end{eqnarray*}
By  the notations  (\ref{eq-2.3.8})  and (\ref{eq-2.3.9}),  we obtain (iii).  
\end{proof}

Taking $f=g$  and $i=j$,   the next corollary  follows from Theorem~\ref{pr-2.3.2} ~(ii). 

\begin{corollary} For any real function $f$  on $V$ and any $0\leq i\leq n$,  we  have 
\begin{eqnarray*}
\partial_i^f\circ  d_i^f= ||f||_2^2 {\rm~id}.  
\end{eqnarray*} 
\qed
\end{corollary}

 \smallskip

\subsection{The Weighted Boundary Operators and the Weighted Co-Boundary Operators  for Paths}\label{ss-2.3}

Let  $f:  V\longrightarrow \mathbb{R}$  be a real function on $V$.  Let $n\geq 0$.   In this subsection, we define  the $f$-weighted boundary operator as the  sum of the $f$-weighted face maps and define the $f$-weighted co-boundary operator as the   sum of the $f$-weighted co-face maps.   Then we   prove the anti-commutative properties for the  weighted boundary operators and  the   weighted co-boundary operators.

\smallskip

\begin{definition}\label{def-2.50}
We define the  {\it $f$-weighted boundary vector}   as a linear map 
\begin{eqnarray*}
\sum_{v\in V} f(v)\frac{\vec\partial}{\partial v}:~~~ \Lambda_n(V)\longrightarrow  \prod_{n+1}\Lambda_{n-1}(V).    
\end{eqnarray*}
For simplicity,  we use the   notation 
\begin{eqnarray*}
\vec\partial^f  :=\sum_{v\in V} f(v)\frac{\vec\partial}{\partial v}.   
\end{eqnarray*}
With the help of (\ref{eq-2.3.8})  we  can write  the $f$-weighted boundary vector as 
\begin{eqnarray}\label{eq-2.5.1}
\vec\partial^f=(\partial_0^f, \partial_1^f,\ldots, \partial_{n}^f).  
\end{eqnarray}
\end{definition}

\begin{definition}\label{def-2.53} 
Taking the   sums  of the coordinates  in  the $f$-weighted boundary vector (\ref{eq-2.5.1}),  we     define the {\it $f$-weighted boundary operator} as 
\begin{eqnarray*}
\partial^f= \sum_{i=0}^n   \partial_i^f:~~~\Lambda_n(V)\longrightarrow \Lambda_{n-1}(V).  
\end{eqnarray*}
\end{definition}

\begin{definition}\label{def-2.51}
We  define  the  {\it $f$-weighted co-boundary vector}   as a linear map 
\begin{eqnarray*}
\sum_{v\in V} f(v)\vec{d} v:~~~ \Lambda_n(V)\longrightarrow  \prod_{n+1}\Lambda_{n+2}(V).    
\end{eqnarray*}
For simplicity,  we use the   notation 
\begin{eqnarray*}
   \vec  d^ f  :=\sum_{v\in V} f(v)\vec{d} v.   
\end{eqnarray*}
With the help of (\ref{eq-2.3.9})  we  can write  the $f$-weighted co-boundary vector as 
\begin{eqnarray}\label{eq-2.5.2}
\vec  d^f=(d_0^f,d_1^f,\ldots, d_{n+1}^f). 
\end{eqnarray}  
\end{definition}

\begin{definition}\label{def-2.55}
Taking the   sums  of the coordinates  in   the $f$-weighted co-boundary  vector   (\ref{eq-2.5.2}),  we    define the {\it $f$-weighted co-boundary operator} as 
\begin{eqnarray*}
 d ^f= \sum_{i=0}^{n+1}  d_i^f:~~~\Lambda_n(V)\longrightarrow \Lambda_{n+1}(V).   
\end{eqnarray*}  
\end{definition}

\begin{definition}\label{def-2.56}
As a particular case of Definition~\ref{def-2.53} and Definition~\ref{def-2.55},   we take any $v\in V$  and   take $f$ to be the characteristic function $\chi_v$ of $v$  given by  
\begin{eqnarray*}
\chi_v(u)=\delta(v,u), ~~~  u\in V.  
\end{eqnarray*}
For any $v\in V$,  we let 
\begin{eqnarray}\label{eq-2.5.5}
\frac{\partial}{\partial v}= \partial^{\chi_v}, ~~~~~~~~~~~~  d v=d^{\chi_v}.  
\end{eqnarray}
\end{definition}

\begin{remark}
We  give a remark on Definition~\ref{def-2.56}.  
It is direct to see that  for any elementary $n$-path $v_0v_1\ldots v_n$  on $V$,  we have 
\begin{eqnarray*}
\frac{\partial}{\partial v} (v_0v_1 \ldots v_n)&=&\sum_{i=0}^n (-1)^i  \delta(v,v_i) v_0\ldots \widehat{v_i}\ldots v_n,
\\
 dv (v_0v_1\ldots v_{n})&=&\sum_{i=0}^{n+1} (-1)^i v_0v_1\ldots v_{i-1} v v_i v_{i+1}\ldots v_{n}.
\end{eqnarray*}
\end{remark}

By Definition~\ref{def-2.53},  Definition~\ref{def-2.55}    and Definition~\ref{def-2.56},    we  have the next lemma.  
 
\begin{lemma}
\label{le-2.5.x}
For any $u,v\in V$  we  have 
\begin{eqnarray}
\label{eq-2.5.7}
\frac{\partial}{\partial u}\circ\frac{\partial}{\partial v} =-\frac{\partial}{\partial v}\circ \frac{\partial}{\partial u}, ~~~~~~~~~~~~  d u \circ  dv=-dv\circ du.  
\end{eqnarray}
\end{lemma}
\begin{proof}
For any $n\geq 0$ and any elementary $n$-path $v_0v_1\ldots v_n\in \Lambda_n(V)$,  we have 
\begin{eqnarray*}
\frac{\partial}{\partial u}\circ\frac{\partial}{\partial v}(v_0v_1\ldots v_n)&=&\frac{\partial}{\partial u}\Big(\sum_{j=0}^n (-1)^j\delta(v,v_j)v_0 \ldots \widehat{v_j} \ldots v_n\Big)\\
&=&\sum_{j=0}^n (-1)^j\delta(v,v_j)\frac{\partial}{\partial u}(v_0 \ldots \widehat{v_j} \ldots v_n)\\
&=&\sum_{j=0}^n (-1)^j\delta(v,v_j)\sum_{i=0}^{j-1} (-1)^i \delta(u,v_i)(v_0 \ldots \widehat{v_i}\ldots \widehat{v_j} \ldots v_n)\\
&& +\sum_{j=0}^n(-1)^j\delta(v,v_j)\sum_{i=j+1}^{n} (-1)^{i-1} \delta(u,v_i)(v_0 \ldots \widehat{v_i}\ldots \widehat{v_j} \ldots v_n)\\
&=&\sum_{0\leq i<j\leq n}(-1)^{i+j}\delta(u,v_i)\delta(v,v_j)(v_0 \ldots \widehat{v_i}\ldots \widehat{v_j} \ldots v_n)\\
&&+ \sum_{0\leq j<i\leq n}(-1)^{i+j-1}\delta(u,v_i)\delta(v,v_j)(v_0 \ldots \widehat{v_j}\ldots \widehat{v_i} \ldots v_n) 
\end{eqnarray*}
and 
\begin{eqnarray*}
\frac{\partial}{\partial v}\circ\frac{\partial}{\partial u}(v_0v_1\ldots v_n)&=&
\sum_{0\leq j<i\leq n}(-1)^{i+j}\delta(u,v_i)\delta(v,v_j)(v_0 \ldots \widehat{v_i}\ldots \widehat{v_j} \ldots v_n)\\
&&+ \sum_{0\leq i<j\leq n}(-1)^{i+j-1}\delta(u,v_i)\delta(v,v_j)(v_0 \ldots \widehat{v_j}\ldots \widehat{v_i} \ldots v_n). 
\end{eqnarray*}
Therefore,  for any elementary $n$-path $v_0v_1\ldots v_n$  on $V$,  we have
\begin{eqnarray*}
\frac{\partial}{\partial u}\circ\frac{\partial}{\partial v}(v_0v_1\ldots v_n)=-\frac{\partial}{\partial v}\circ\frac{\partial}{\partial u}(v_0v_1\ldots v_n).
\end{eqnarray*}
Consequently, by the   linear property of $\frac{\partial}{\partial u}\circ\frac{\partial}{\partial v}$ and $\frac{\partial}{\partial v}\circ\frac{\partial}{\partial u}$,  we   obtain the  first identity in (\ref{eq-2.5.7}).

Similarly,   For any $n\geq 0$ and any elementary $n$-path $v_0v_1\ldots v_n\in \Lambda_n(V)$,  we have  
\begin{eqnarray*}
du \circ  dv(v_0v_1\ldots v_n)&=&du\Big(\sum_{i=0}^{n+1} (-1)^i v_0\ldots v_{i-1} v v_i\ldots v_n\Big)\\
&=&\sum_{i=0}^{n+1} (-1)^i du(v_0\ldots v_{i-1} v v_i\ldots v_n)\\
&=&\sum_{i=0}^{n+1} (-1)^i \Big(\sum_{j=0}^{i-1}(-1)^j v_0\ldots v_{j-1}  u v_j \ldots v_{i-1} v  v_i \ldots v_n\\
&&+(-1)^i v_0\ldots v_{i-1}uvv_i\ldots v_n + (-1)^{i+1}v_0\ldots v_{i-1} v u v_i\ldots v_n\\
&&+  \sum_{j=i+1}^{n+1} (-1)^{j+1} v_0\ldots v_{i-1} v v_i \ldots v_{j-1} u v_j\ldots v_n\Big)
\end{eqnarray*}
and  
\begin{eqnarray*}
dv \circ  du(v_0v_1\ldots v_n)&=&\sum_{i=0}^{n+1}(-1)^i \Big(\sum_{j=0}^{i-1}(-1)^j v_0\ldots v_{j-1}  v v_j \ldots v_{i-1} u  v_i \ldots v_n\\
&&+(-1)^i v_0\ldots v_{i-1}vuv_i\ldots v_n + (-1)^{i+1}v_0\ldots v_{i-1} u v v_i\ldots v_n\\
&&+  \sum_{j=i+1}^{n+1} (-1)^{j+1} v_0\ldots v_{i-1} u v_i \ldots v_{j-1} v v_j\ldots v_n\Big).
\end{eqnarray*}
Thus 
\begin{eqnarray*}
du \circ dv(v_0v_1\ldots v_n)=
-dv \circ  du(v_0v_1\ldots v_n)  
\end{eqnarray*}
 for any $n\geq 0$ and any elementary $n$-path $v_0v_1\ldots v_n\in \Lambda_n(V)$.   Consequently,   by the linear property of  $du \circ dv(v_0v_1\ldots v_n)$  and $dv \circ  du(v_0v_1\ldots v_n)$,   we obtain   the second identity in (\ref{eq-2.5.7}).  
\end{proof}

The next corollary is a re-statement of Lemma~\ref{le-2.5.x}. 

\begin{corollary}\label{co-2.5.y}
For any  $u,v\in V$ we  have 
\begin{eqnarray*}
\sum_{i,j=0}^n\frac{\partial_i}{\partial v}\circ\frac{\partial_j}{\partial u} =-\sum_{i,j=0}^n  \frac{\partial_j}{\partial u}\circ\frac{\partial_i}{\partial v}, ~~~~~~~~~~~~
\sum_{i,j=0}^n d_i v \circ  d_j u=-\sum_{i,j=0}^n  d_j u\circ  d_i v. 
\end{eqnarray*}
\end{corollary}

\begin{proof}
The first identity follows from Lemma~\ref{le-2.5.x} by  
\begin{eqnarray*}
&&\sum_{i,j=0}^n\frac{\partial_i}{\partial v}\circ\frac{\partial_j}{\partial u}=\Big(\sum_{i=0}^n \frac{\partial_i}{\partial v}\Big) \circ  \Big(\sum_{j=0}^n \frac{\partial_j}{\partial u}\Big)
= \frac{\partial}{\partial v}\circ\frac{\partial}{\partial u} \\
&=& -\frac{\partial}{\partial u}\circ\frac{\partial}{\partial v}=-\Big(\sum_{j=0}^n \frac{\partial_j}{\partial u}\Big) \circ  \Big(\sum_{i=0}^n \frac{\partial_i}{\partial v}\Big)=-\sum_{i,j=0}^n\frac{\partial_j}{\partial u}\circ\frac{\partial_i}{\partial v}. 
\end{eqnarray*}
The second identity follows from Lemma~\ref{le-2.5.x} by
\begin{eqnarray*}
&&\sum_{i,j=0}^n d_i v\circ   d_j u=\Big(\sum_{i=0}^n  d_i v\Big) \circ  \Big(\sum_{j=0}^nd_j u\Big)
= d v\circ  du \\
&=& -du \circ  dv =-\Big(\sum_{j=0}^n d_j u\Big) \circ  \Big(\sum_{i=0}^n  d_i v \Big)=-\sum_{i,j=0}^nd_j u\circ  d_i v. 
\end{eqnarray*}
  The corollary is proved. 
\end{proof}

 With the help of Corollary~\ref{co-2.5.y},  we generalize Lemma~\ref{le-2.5.x}   in the next proposition.  
\begin{proposition}[The anti-commutative properties]\label{pr-2.5.5}
Let $f$ and $g$ be two real functions on $V$.  Then 
\begin{eqnarray*}
\partial^{f}\circ \partial^g =-\partial^g\circ \partial ^f, ~~~~~~~~~~~~  d^f\circ d^g =- d^g\circ d^f. 
\end{eqnarray*}
In particular,  taking $f= \chi_u$ and $g=\chi_v$,  we obtain (\ref{eq-2.5.7}). 
\end{proposition}

\begin{proof}
Let $n\geq 0$.   Since
\begin{eqnarray*}
\partial^f= \sum_{i=0}^n \partial_i^f = \sum_{i=0}^n \sum_{v\in V} f(v) \frac{\partial_i}{\partial v}, 
\end{eqnarray*}
by a straight-forward calculation we have  
\begin{eqnarray*}
\partial^f\circ \partial^g&=& \Big(\sum_{i=0}^n \sum_{v\in V} f(v) \frac{\partial_i}{\partial v}\Big)\circ\Big(\sum_{j=0}^n \sum_{u\in V} g(u) \frac{\partial_j}{\partial u}\Big)\\
&=& \sum_{v,u\in V} f(v)g(u) \sum_{i,j=0}^n\frac{\partial_i}{\partial v}\circ\frac{\partial_j}{\partial u} 
\end{eqnarray*}
and
\begin{eqnarray*}
\partial^g\circ \partial^f&=& \Big(\sum_{j=0}^n \sum_{u\in V} g(u) \frac{\partial_i}{\partial u}\Big)\circ\Big(\sum_{i=0}^n \sum_{v\in V} f(v) \frac{\partial_i}{\partial v}\Big)\\
&=&\sum_{v,u\in V} f(v)g(u)\sum_{i,j=0}^n  \frac{\partial_j}{\partial u}\circ\frac{\partial_i}{\partial v}.   
\end{eqnarray*}
Hence with the help of the first identity in Corollary~\ref{co-2.5.y},  we have 
$
\partial^{f}\circ \partial^g =-\partial^g\circ \partial ^f$.   Similarly, since
\begin{eqnarray*}
d^f= \sum_{i=0}^n d_i^f = \sum_{i=0}^n \sum_{v\in V} f(v) d_i v, 
\end{eqnarray*}
by a straight-forward calculation we have  
\begin{eqnarray*}
d^f\circ d^g&=& \Big(\sum_{i=0}^n \sum_{v\in V} f(v) d_i v\Big)\circ\Big(\sum_{j=0}^n \sum_{u\in V} g(u)d_j u\Big)\\
&=& \sum_{v,u\in V} f(v)g(u) \sum_{i,j=0}^nd_i v\circ d_j u 
\end{eqnarray*}
and
\begin{eqnarray*}
d^g\circ d^f&=& \Big(\sum_{j=0}^n \sum_{u\in V} g(u) d_j u\Big)\circ\Big(\sum_{i=0}^n \sum_{v\in V} f(v) d_i v\Big)\\
&=&\sum_{v,u\in V} f(v)g(u)\sum_{i,j=0}^n  d_j u\circ  d_i v.   
\end{eqnarray*}
 Hence with the help of the second identity in Corollary~\ref{co-2.5.y},  we have 
$
d^{f}\circ d^g =-d^g\circ d^f$.
 The proposition follows.  
\end{proof}


With respect to the canonical inner product   $\langle~,~\rangle$ on $\Lambda_*(V)$  given by 
\begin{eqnarray}\label{eq-inn}
\langle  v_0v_1\ldots v_n, u_0u_1\ldots u_n\rangle =\prod_{i=0}^n \delta(v_i,u_i),  ~~~n=0,1,2,\ldots, 
\end{eqnarray}
the linear operator $dv$  is adjoint  to $\frac{\partial }{\partial v}$ for any $v\in V$.   Precisely,  we   have the next lemma.   
 
 \begin{lemma}\label{le-adj}
 For any $n\geq 1$,  any $\xi\in \Lambda_{n-1}(V)$,   and any $\eta\in \Lambda_{n}(V)$,   we  have 
\begin{eqnarray}\label{eq-2.in}
\langle  \frac{\partial}{\partial v}(\eta),\xi\rangle =\langle \eta, dv(\xi)\rangle.  
\end{eqnarray}
 \end{lemma}
 \begin{proof}
 we take $\eta$  to be an elementary $n$-path $v_0v_1\ldots v_n\in \Lambda_n(V)$  and  take  $\xi$  to be an elementary $(n-1)$-path $u_0u_1\ldots u_{n-1}\in \Lambda_{n-1}(V)$.  Then 
\begin{eqnarray*}
 \langle \frac{\partial}{\partial v}(v_0v_1\ldots v_n),   u_0u_1\ldots u_{n-1} \rangle 
&=&\langle \sum _{i=0}^n (-1)^i \delta(v,v_i)v_0\ldots \widehat{v_i}\ldots v_n, u_0u_1\ldots u_{n-1}\rangle\\
&=&\sum _{i=0}^n (-1)^i \delta(v,v_i)\prod_{j=0}^{i-1}\delta(v_j,u_j) \prod _{j=i}^{n-1}\delta(v_{j+1},u_{j}). 
\end{eqnarray*}
Consequently,  if we use $(\frac{\partial}{\partial v})^*$  to denote the adjoint linear operator of  $\frac{\partial}{\partial v}$,  then  we  have 
\begin{eqnarray*}
(\frac{\partial}{\partial v})^* (u_0u_1\ldots u_{n-1})&=&\sum_{v_0,v_1,\ldots, v_n\in V}  \langle v_0v_1\ldots v_n, (\frac{\partial}{\partial v})^* (u_0u_1\ldots u_{n-1}) \rangle  v_0v_1\ldots v_n \nonumber\\
&=&\sum_{v_0,v_1,\ldots, v_n\in V}  \langle \frac{\partial}{\partial v}  (v_0v_1\ldots v_n),  u_0u_1\ldots u_{n-1} \rangle  v_0v_1\ldots v_n \nonumber\\
&=&\sum_{v_0,v_1,\ldots, v_n\in V}\Big(\sum _{i=0}^n (-1)^i \delta(v,v_i)\prod_{j=0}^{i-1}\delta(v_j,u_j) \prod _{j=i}^{n-1}\delta(v_{j+1},u_{j})\Big) v_0v_1\ldots v_n \nonumber\\
&=&\sum_{i=0}^n (-1)^i \Big( \sum_{v_0,v_1,\ldots, v_n\in V} \delta(v,v_i)\prod_{j=0}^{i-1}\delta(v_j,u_j) \prod _{j=i}^{n-1}\delta(v_{j+1},u_{j})\Big)v_0v_1\ldots v_n \nonumber\\ 
&=&\sum_{i=0}^n (-1)^i u_0u_1\ldots u_{i-1} v u_i u_{i+1}\ldots u_{n-1} \nonumber\\
&=& dv (u_0u_1\ldots u_{n-1}).     
\end{eqnarray*} 
Therefore,    by the linear property of  $(\frac{\partial}{\partial v})^*$  and  $dv$,    we  have  
\begin{eqnarray*}
(\frac{\partial}{\partial v})^*=dv. 
\end{eqnarray*}
The  lemma is proved.  
 \end{proof}

 In general,   the linear operator $d^f$  is adjoint to $\partial^f$  for any real function $f$  on $V$.   The next proposition follows from Lemma~\ref{le-adj}.

 \begin{proposition}[The adjoint property]
 \label{pr-adj}
 Let $f$  be any real function on $V$.  Then for any $n\geq 1$,  any $\xi\in \Lambda_{n-1}(V)$,   and any $\eta\in \Lambda_{n}(V)$,   we  have 
\begin{eqnarray}\label{eq-88.in}
\langle  \partial^f(\eta),\xi\rangle =\langle \eta, d^f(\xi)\rangle.  
\end{eqnarray}
 \end{proposition}
 
 \begin{proof}
 Let $n\geq 1$.   Let $f$  be any real function on $V$.   Let    $\xi\in \Lambda_{n-1}(V)$     and   $\eta\in \Lambda_{n}(V)$.  Then  
 \begin{eqnarray*}
 \langle  \partial^f(\eta),\xi\rangle&=& \langle  \sum_{i=0}^n   \partial_i^f(\eta),\xi\rangle\\
 &=&\Big\langle  \sum_{i=0}^n\Big(\sum_{v\in V}  f(v) \dfrac{\partial_i}{\partial v}\Big) (\eta),\xi\Big\rangle\\
 &=&\Big\langle \sum_{v\in V}  f(v) \sum_{i=0}^n  \Big(\dfrac{\partial_i}{\partial v}\Big)(\eta),\xi\Big\rangle\\
 &=&\Big\langle \sum_{v\in V} f(v)  \dfrac{\partial }{\partial v}(\eta),\xi\Big\rangle\\
 &=&\sum_{v\in V}  f(v)  \langle \eta, dv (\xi) \rangle\\
 &=&\langle \eta,  d^f (\xi)\rangle. 
 \end{eqnarray*}
 We  obtain (\ref{eq-88.in}). 
  \end{proof}

\smallskip

We prove some  Newton-Leibniz-type rules in the next proposition. 

\begin{proposition}[The Newton-Leibniz-type rules]\label{pr-2.5.0}
Let $f$  be a real function on $V$.  
Let $n,m\geq 0$.  Let $\xi\in \Lambda_n(V)$ be given in (\ref{eq-2.0.1}) and let  $\eta\in \Lambda_m(V)$ be given in (\ref{eq-2.0.2}).  Then  
\begin{eqnarray}\label{eq-nl-1}
\partial^f (\xi *\eta)&=&  \partial^f(\xi) *\eta +  (-1)^{n+1}\xi *\partial^f(\eta),\\
 d^f (\xi*\eta)&=&  d^f(\xi)*\eta + (-1)^{n+1} \xi* d^f(\eta). 
 \label{eq-nl-2}
\end{eqnarray}
\end{proposition}
\begin{proof}
By a straight-forward calculation and with the help of (\ref{eq-2.0.3}),  we  have 
\begin{eqnarray*}
\partial^f(v_0v_1\ldots v_n u_0u_1\ldots u_m)&=&\sum_{i=0}^{n+m+1} \partial_i^f(v_0v_1\ldots v_n u_0u_1\ldots u_m) \\
&=&\sum_{i=0}^{n+m+1} \sum_{w\in V} f(w) \frac{\partial_i}{\partial  w}(v_0v_1\ldots v_n u_0u_1\ldots u_m)\\
&=&\sum_{i=0}^{n} \sum_{v\in V} (-1)^i f(v)\delta(v,v_i) (v_0\ldots \widehat{v_i}\ldots  v_n)*( u_0u_1\ldots u_m)\\
&&+   \sum_{i=0}^{m}  \sum_{u\in V}(-1)^{n+i+1}f(u)\delta(u,u_i)(v_0v_1\ldots v_n)  * (u_0 \ldots\widehat{u_i}\ldots u_m)\\
 &=&\sum_{i=0}^{n} \partial_i^f(v_0v_1\ldots v_n) * (u_0u_1\ldots u_m) \\
 &&+ (-1)^{n+1} \sum_{i=0}^{m}  (v_0v_1\ldots v_n)  * \partial_i^f(u_0u_1\ldots u_m)\\
 &=&\partial^f (v_0v_1\ldots v_n) * (u_0u_1\ldots u_m) \\
 &&+ (-1)^{n+1} (v_0v_1\ldots v_n) * \partial^f(u_0u_1\ldots u_m).  
\end{eqnarray*}
Consequently,  
\begin{eqnarray*}
\partial^f (\xi *\eta)&=& \sum_{v_0,v_1,\ldots,v_n\in V;\atop u_0,u_1,\ldots, u_m\in V}  r_{v_0v_1\ldots v_n}  t_{u_0u_1\ldots u_m} \partial^f (  v_0v_1\ldots v_n u_0u_1\ldots u_m)\\
&=& \sum_{v_0,v_1,\ldots,v_n\in V;\atop u_0,u_1,\ldots, u_m\in V}  r_{v_0v_1\ldots v_n}  t_{u_0u_1\ldots u_m}\Big( \partial^f (  v_0v_1\ldots v_n )* (u_0u_1\ldots u_m)\\
&& ~~~~~~~~~~~~~~~~~~~~~~~~~~~~~~~~~~~~~~~~~ +  (-1)^{n+1}  (v_0v_1\ldots v_n) * \partial^f(u_0u_1\ldots u_m)\Big) \\
&=&  \partial^f(\xi) *\eta +  (-1)^{n+1}\xi *\partial^f(\eta).  
\end{eqnarray*}
We obtain (\ref{eq-nl-1}).  On the other hand,  it follows from   a straight-forward calculation  that 
\begin{eqnarray*}
d^f(v_0v_1\ldots v_n u_0u_1\ldots u_n)&=&\sum_{i=0}^{n+1} d_i^f(v_0v_1\ldots v_n) * (u_0u_1\ldots u_n) \\
&&+ (-1)^{n+1} \sum_{i=0}^{m+1}  (v_0v_1\ldots v_n)  * d_i^f(u_0u_1\ldots u_n).  
 \end{eqnarray*}
Thus  by a formal calculation analogous with the above  proof of (\ref{eq-nl-1}),  we  obtain (\ref{eq-nl-2}). 
 The proposition follows.  
\end{proof}

By the end of this subsection,   we summarize  Proposition~\ref{pr-2.5.5} and Proposition~\ref{pr-2.5.0}.    We  give   the  exterior algebra $T_*(V)$ generated by all the weighted boundary operators and the  exterior algebra  $T^*(V)$  generated by all the weighted co-boundary operators on the discrete set $V$,  in the following list.  

\begin{quote}
\begin{itemize}
\item
Let $T_*(V)=\bigoplus_{k\geq 0}  T_k(V)$ be  the exterior algebra  spanned by  $\partial ^f$  for all real functions $f$  on $V$.  Then we  have all of the followings: 
\begin{enumerate}[(i).]
\item
 for each $k\geq 0$ and for each   $\alpha \in T_k(V)$,  we  have a graded linear map $\alpha: \Lambda_n(V)\longrightarrow \Lambda_{n-k}(V)$, $n= 0,1,2\ldots$;    
\item
the exterior product in  $T_*(V)$   is the composition  of the graded linear maps in (i); 
\item
the operations of  $T_*(V)$  on $\Lambda_*(V)$ satisfy the Newton-Leibniz-type law (\ref{eq-nl-1}).   
\end{enumerate}

\item
Let $T^*(V)=\bigoplus_{k\geq 0}  T^k(V)$ be  the exterior algebra  spanned by  $d ^f$  for all real functions $f$  on $V$.  Then we  have all of the followings: 
\begin{enumerate}[(i)'.]
\item
 for each $k\geq 0$ and for each   $\omega \in T^k(V)$,  we  have a graded linear map $\omega: \Lambda_n(V)\longrightarrow \Lambda_{n+k}(V)$, $n= 0,1,2\ldots$;    
\item
the exterior product in  $T^*(V)$   is the composition  of the graded linear maps in (i)'; 
\item
the operations of  $T^*(V)$  on $\Lambda_*(V)$ satisfy the Newton-Leibniz-type law (\ref{eq-nl-2}).   
\end{enumerate}
\end{itemize}
\end{quote}

\smallskip

\subsection{A Contrast with The Usual Simplicial Identities }\label{ss2.4}

In this subsection,  we re-state the usual simplicial identities of  simplicial sets,  in    contrast  with Theorem~\ref{pr-2.3.2}.  This subsection is supplementary to Subsection~\ref{ss2.2} - Subsection~\ref{ss-2.3}.  

\smallskip

Consider the special function $f\equiv 1$, that is,  $f(v)$ takes the constant value $1$  for all   $v\in V$.  We denote  this function $f$ as $1$.  We  denote  the corresponding $f$-weighted  face maps as  $\partial_i ^1$   for  $0\leq i\leq n$  and denote the corresponding $f$-weighted  co-face maps as $d_i^1$   for   $0\leq i\leq n+1$.  Let $v_0v_1\ldots v_n$ be an elementary $n$-path on $V$.   It follows that 
\begin{eqnarray*}
\partial^1_i  (v_0v_1\ldots v_n)= (-1)^i v_0\ldots \widehat{v_i}\ldots v_n.  
\end{eqnarray*}
Thus  the linear map $(-1)^i \partial^1_i: \Lambda_n(V)\longrightarrow \Lambda_{n-1}(V)$  given by  
\begin{eqnarray}\label{eq-2.6.1}
(-1)^i \partial^1_i(v_0v_1\ldots v_n)=   v_0\ldots \widehat{v_i}\ldots v_n
\end{eqnarray}
 is the  usual  face map of simplicial sets (cf. \cite[pp. 110-111]{curtis}).   
 Moreover,  for $0\leq i\leq n$,  we use the linear map $s_i:  \Lambda_n(V)\longrightarrow \Lambda_{n+1}(V)$  to denote the $i$-th degeneracy (cf. \cite[pp. 110-111]{curtis}) given by 
\begin{eqnarray}\label{eq-2.6.2}
s_i(v_0v_1\ldots v_n)= v_0\ldots v_{i-1} v_i v_i v_{i+1}\ldots v_n.   
\end{eqnarray}
Then for any elementary $n$-path $v_0v_1\ldots v_n$ on $V$,  we  have 
\begin{eqnarray}\label{eq-2.6.3}
s_i(v_0v_1\ldots v_n)&=&(-1)^i d_iv_i (v_0v_1\ldots v_n) \nonumber\\
&=&(-1)^i d_i^{\chi_{v_i}}(v_0v_1\ldots v_n)
\end{eqnarray}
and 
\begin{eqnarray}
s_i(v_0v_1\ldots v_n)&=&(-1)^{i+1} d_{i+1} v_i (v_0v_1\ldots v_n)\nonumber\\
&=&(-1)^{i+1} d_{i+1}^{\chi_{v_i}}(v_0v_1\ldots v_n).  
\label{eq-2.6.5}  
\end{eqnarray}
 With the help of (\ref{eq-2.6.1}) - (\ref{eq-2.6.5}),   the   simplicial identities  (cf. \cite[p. 110]{curtis})   can be re-stated   in the next proposition. 
\begin{proposition}[The usual simplicial identities,   in    contrast  with Theorem~\ref{pr-2.3.2}]
\label{pr-2.5.1}
Let $n\geq 0$ and $0\leq i,j\leq n$.  Then we  have  the first simplicial identity 
\begin{eqnarray*}
\partial^1_i \partial^1_j=-\partial^1_{j-1}\partial^1_i,   ~~~~~~~~~~~~ i<j,
\end{eqnarray*}
 the second simplicial identity 
\begin{eqnarray*}
\partial^1_i s_j =\begin{cases}
s_{j-1} \partial^1_i,  & i<j,\\
{\rm id},  & i=j, j+1,\\
-s_j \partial_{i-1},  & i>j+1, 
\end{cases}
\end{eqnarray*}
and the third simplicial identity 
\begin{eqnarray*}
s_i s_j= s_{j+1} s_i,  ~~~~~~~~~~~~  i\leq j. 
\end{eqnarray*}
\qed
\end{proposition}

\smallskip

\section{Simplicial-Like Identities for Regular Paths on Discrete Sets}

Let $V$  be a discrete set.  
 Let $n\geq 0$.  An elementary $n$-path $v_0\ldots v_n$ on $V$ is called {\it regular} if $v_{i-1}\neq v_i$  for all $1\leq i\leq n$   and is called {\it irregular} otherwise  (cf.  \cite[Definition~2.7]{lin2}).  For each $n\geq 0$,  let $I_n(V)$ be the subspace of $\Lambda_n(V)$ spanned by all the  irregular elementary $n$-paths on $V$.  We  have a graded subspace 
 \begin{eqnarray*}
 I_*(V)=\bigoplus_{n=0}^\infty I_n(V)
 \end{eqnarray*}
 of $\Lambda_*(V)$. 
  Consider the quotient space $\mathcal{R}_n(V)=
\Lambda_n(V)/I_n(V)$. Then $\mathcal{R}_n(V)$  is the vector space spanned by all the regular elementary $n$-paths on $V$  (cf. \cite[Definition~2.8]{lin2}).  An element in $\mathcal{R}_n(V)$  is called a {\it  regular $n$-path}  on $V$.   We  take the direct sum
\begin{eqnarray*}
\mathcal{R}_*(V)=\bigoplus_{n= 0}^\infty \mathcal{R}_n(V). 
\end{eqnarray*}

\smallskip

\subsection{The Weighted Face Maps and the Weighted Co-Face Maps  for Regular Paths}

 Let $f$  be a real function on $V$.    In  this subsection,   we  define the  $f$-weighted face maps and the $f$-weighted co-face maps  for regular paths on $V$ and prove some simplicial-like identities. 
 
 \smallskip
 
 \begin{definition}\label{def-3.a1}
   For each  $0\leq i\leq n$,   it follows from the argument in \cite[Subsection~2.3, Regular paths]{lin2} that  if we modulo the terms in $I_{n-1}(V)$ of the image  of $\partial_{i}^f: \Lambda_{n}(V)\longrightarrow\Lambda_{n-1}(V)$  \footnote{by saying "modulo the terms in $I_{n-1}(V)$ of the image  of $\partial_{i}^f$",  it means  that  we  take the canonical  projection  from the image  of $\partial_{i}^f$  in $\Lambda_{n-1}(V)$   to the orthogonal complement  of $I_{n-1}(V)$ in $\Lambda_{n-1}(V)$ with respect to the inner product  (\ref{eq-inn}).  },  then the $f$-weighted face  map $\partial^f_i$ from $\Lambda_n(V)$ to $\Lambda_{n-1}(V)$   induces a   linear map  
 \begin{eqnarray*}
 \tilde \partial^f_i:=\partial^f_i/I_*(V): ~~~\mathcal{R}_n(V)\longrightarrow \mathcal{R}_{n-1}(V).  
 \end{eqnarray*}
 We  call  $\tilde \partial^f_i$  the {\it $i$-th $f$-weighted regular face map}.  \end{definition}

 We give the explicit expression for the $i$-th $f$-weighted regular face map $\tilde \partial^f_i$ defined in Definition~\ref{def-3.a1}.    Let $v_0v_1\ldots v_n$  be an arbitrary  regular 
  elementary $n$-path  on $V$.   For any $u,v\in V$,  we let
\begin{eqnarray*}
\epsilon(u,v)=1-\delta(u,v). 
\end{eqnarray*}
  With the help of  the first formula in Proposition~\ref{le-2.3.99},   we  have 
  \begin{eqnarray}\label{eq-regular-1}
  \tilde \partial^f_i (v_0v_1\ldots v_n)=\begin{cases}
     f(v_0) v_1  \ldots v_n, &i=0,\\
  (-1)^i \epsilon(v_{i-1},v_{i+1}) f(v_i) v_0\ldots \widehat{v_i} \ldots v_n, &1\leq i\leq n-1,\\
      (-1)^n f(v_n) v_0 \ldots v_{n-1}, &i=n.
       \end{cases}  
  \end{eqnarray}
For convenience,  for all $0\leq i\leq n$  we write (\ref{eq-regular-1})  as 
\begin{eqnarray*}
  \tilde \partial^f_i (v_0v_1\ldots v_n)=(-1)^i \epsilon(v_{i-1},v_{i+1}) f(v_i) v_0\ldots \widehat{v_i} \ldots v_n 
\end{eqnarray*}
for short,  by an abuse of the following notations 
\begin{eqnarray*}
\epsilon(v_{-1},v_1)=\epsilon(v_{n-1},v_{n+1})=1.   
\end{eqnarray*}

\begin{definition}\label{def-3.a2}
For any   $0\leq i\leq n+1$,  if we modulo the terms in $I_*(V)$ of the image,   then  the $f$-weighted co-face  map $d^f_i$ from $\Lambda_n(V)$ to $\Lambda_{n+1}(V)$   induces a   linear map  
 \begin{eqnarray*}
\tilde d^f_i:=d^f_i/I_*(V): ~~~\mathcal{R}_n(V)\longrightarrow \mathcal{R}_{n+1}(V).  
 \end{eqnarray*}
 We  call  $\tilde d^f_i$  the {\it $i$-th $f$-weighted regular co-face map}.  
 \end{definition}
 \begin{remark}
 Note that the usual degeneracy $s_i$  given by (\ref{eq-2.6.2})  induces an identically-zero map  $s_i/I_*(V)=0$ from $\mathcal{R}_n(V)$ to $\mathcal{R}_{n+1}(V)$. 
 \end{remark}
 
 We give the explicit expression for the $i$-th $f$-weighted regular co-face map $\tilde   d^f_i$  defined in Definition~\ref{def-3.a2}.     Let $v_0v_1\ldots v_n$  be an arbitrary  regular 
  elementary $n$-path  on $V$.    With the help of the second   formula in Proposition~\ref{le-2.3.99},  we  have 
 \begin{eqnarray}\label{eq-regular-2}
\tilde d^f_i (v_0v_1\ldots v_n)=\begin{cases}
\sum_{v\in V\setminus\{v_0\}}  f(v)  v v_0\ldots v_n, &i=0,\\
\sum_{v\in V\setminus \{v_{j-1},v_j\}} (-1)^i   f(v) v_0\ldots  v_{i-1} v v_i\ldots v_n, 
&0\leq i\leq n,\\
\sum_{v\in V\setminus\{ v_n\}}  (-1)^{n+1} f(v)  v_0\ldots v_n v,  & i=n+1. 
\end{cases}
\end{eqnarray}
For  convenience,  for all $0\leq i\leq n+1$  we  write (\ref{eq-regular-2})  as  
\begin{eqnarray*}
\tilde d^f_i (v_0v_1\ldots v_n)=\sum_{v\neq v_{i-1},v_i} (-1)^i   f(v) v_0\ldots  v_{i-1} v v_i\ldots v_n
\end{eqnarray*}
for short, by an abuse of the following notations
\begin{eqnarray*}
v\neq v_{-1}, v_0\Longleftrightarrow  v\neq v_0, ~~~~~~~~~~~~  
v\neq v_{n},  v_{n+1}\Longleftrightarrow v\neq v_n. 
\end{eqnarray*}

It follows from Theorem~\ref{pr-2.3.2},  Definition~\ref{def-3.a1}, (\ref{eq-regular-1}),    Definition~\ref{def-3.a2}  and (\ref{eq-regular-2})  that the regular face maps and regular co-face maps  also  satisfy some   simplicial-like identities partially:  
\begin{theorem}[Main Result II: The simplicial-like identities for regular paths  on discrete sets]
\label{pr-3.1.1}
Let  $f$  and $g$ be  two real functions on $V$.   Then for any $n\geq 0$,  we  have 
\begin{enumerate}[(i).]
\item
$\tilde\partial_i^f\circ \tilde\partial_j^g=-\tilde\partial_{j-1}^g\circ \tilde\partial_i^f$ for any $i\leq j-2$; 

\item
$
\tilde\partial_i ^f\circ \tilde d_j^g= \begin{cases}
- \tilde d_{j-1}^g \circ\tilde \partial_i^f,  & i\leq j-2,\\
\langle f,g\rangle^{V\setminus\{v_{j-1}, v_j\}}  {\rm~ id},   & i=j, \\
- \tilde d_j^g \circ \tilde \partial_{i-1}^f,  &  i\geq j+2; 
\end{cases}
$

\item
$\tilde d_i^f \circ  \tilde d_j^g= - \tilde d_{j+1}^g \circ  \tilde d_i^f$  for $i\leq j-1$.    
\end{enumerate}
\end{theorem}

\begin{proof}
Let $v_0v_1\ldots v_n$  be a regular elementary $n$-path on $V$.  By the definition of the regular paths on $V$,  for any $0\leq i\leq n$  we have 
\begin{eqnarray}\label{eq-3.basic}
\epsilon(v_i,v_{i+1})=\epsilon(v_i,v_{i-1})=1. 
\end{eqnarray}
We prove (i), (ii), and  (iii)  separately.

(i). Suppose  $i<j$.  By a straight-forward calculation,   we  have 
\begin{eqnarray*}
&&\tilde\partial^f_i\circ \tilde\partial^g_j(v_0v_1\ldots v_n)\\
&=&\begin{cases}
(-1)^{i+j}  f(v_i) g(v_j) \epsilon(v_{j-1}, v_{j+1}) \epsilon(v_{i-1}, v_{i+1}) v_0\ldots \widehat{v_i}\ldots \widehat{v_j} \ldots v_n, & i\leq j-2, \\
-f(v_{j-1}) g(v_j) \epsilon(v_{j-1}, v_{j+1}) \epsilon(v_{j-2}, v_{j+1}) v_0\ldots \widehat{v_{j-1}} \widehat{v_j} \ldots v_n,  &i=j-1. 
\end{cases}
\end{eqnarray*}
On the other hand,  
\begin{eqnarray*}
&&\tilde\partial^g_{j-1}\circ \tilde\partial^f_i(v_0v_1\ldots v_n)\\
&=&\begin{cases}
(-1)^{i+j-1}  f(v_i) g(v_j) \epsilon(v_{j-1}, v_{j+1}) \epsilon(v_{i-1}, v_{i+1}) v_0\ldots \widehat{v_i}\ldots \widehat{v_j} \ldots v_n, & i\leq j-2, \\
f(v_{j-1}) g(v_j) \epsilon(v_{j-2}, v_{j+1}) \epsilon(v_{j-2}, v_{j}) v_0\ldots \widehat{v_{j-1}} \widehat{v_j} \ldots v_n,  &i=j-1. 
\end{cases}
\end{eqnarray*}
Therefore,  we obtain (i)  for $i\leq j-2$.

    (ii).    By a straight-forward calculation,   for any possible $i$ and $j$  we have 
\begin{eqnarray*}
&&\tilde \partial_i^f \circ \tilde d_j^g(v_0v_1\ldots v_n)\\
&=&\tilde \partial_i^f\Big(\sum_{v\neq v_{j-1},v_j} (-1)^j   g(v) v_0\ldots  v_{j-1} v v_j\ldots v_n\Big)\\
&=&\begin{cases}
\sum_{v\neq v_{j-1},v_j} (-1)^{j+i}  g (v) f(v_i)\epsilon(v_{i-1},v_{i+1})  v_0\ldots  \widehat{v_i}  \ldots  v_{j-1} v v_j\ldots v_n,  &i\leq j-2,\\
-\sum_{v\neq v_{j-1},v_j}   g (v) f(v_{j-1})\epsilon(v_{j-2},v)  v_0\ldots  \widehat{v_{j-1}}  v v_j\ldots v_n,  &i=j-1,\\
\sum_{v\neq v_{j-1},v_j}    g (v) f(v)   v_0 \ldots v_n,  &i=j,\\
-\sum_{v\neq v_{j-1},v_j}   g(v) f(v_j) \epsilon(v,v_{j+1}) v_0\ldots v_{j-1} v \widehat{v_j} \ldots v_n, 
&i=j+1,\\
\sum_{v\neq v_{j-1},v_j}  (-1)^{j+i} g(v) f(v_{i-1})\epsilon(v_{i-2},v_i) v_0\ldots v_{j-1} v v_j\ldots \widehat{v_{i-1}}\ldots v_n,  &i\geq j+2. 
\end{cases}
\end{eqnarray*}
Therefore, we obtain (ii) for the case $i=j$.  Moreover,  for $i<j$  we  have 
\begin{eqnarray*}
&&\tilde d_{j-1}^g \circ\tilde \partial_i^f(v_0v_1\ldots v_n)\\
&=&
\tilde d_{j-1}^g \Big((-1)^i \epsilon(v_{i-1}, v_{i+1}) f(v_i) v_0\ldots \widehat{v_i} \ldots v_n\Big)\\
&=&\begin{cases}
\sum_{v\neq v_{j-1}, v_j}(-1)^{i+j-1} \epsilon(v_{i-1}, v_{i+1}) f(v_i) g(v)  v_0\ldots \widehat{v_i}\ldots v_{j-1} v v_j\ldots v_n, &j\geq i+2,\\
\sum_{v\neq v_i, v_{i+1}} \epsilon(v_{i-1}, v_{i+1})  f(v_i) g(v) v_0\ldots \widehat{v_i} v v_{i+1} \ldots v_n, & j=i+1.  
\end{cases}
\end{eqnarray*}
Therefore,   we obtain (ii)  for the case  $i\leq j-2$.  
Furthermore,  for $i>j$ we   have 
\begin{eqnarray*}
&& \tilde d_j^g \circ \tilde \partial_{i-1}^f(v_0v_1\ldots v_n)\\
&=&(-1)^{i-1} \epsilon(v_{i-2},v_i) f(v_{i-1}) \tilde d^g_j (v_0\ldots \widehat{v_{i-1} }\ldots v_n)\\
&=&
\begin{cases}
 \sum_{v\neq  v_{j-1}, v_j }(-1)^{i+j-1} \epsilon(v_{i-2}, v_i) 
f(v_{i-1})  g(v) v_0\ldots v_{j-1} v v_j \ldots \widehat {v_{i-1}} \ldots v_n,  & j\leq i-2,\\
\sum_{v\neq v_{i-2},  v_i}  \epsilon(v_{i-2}, v_i) f(v_{i-1}) g(v) v_0\ldots v_{i-2} v \widehat{v_{i-1}} \ldots v_n, & j=i-1.  
\end{cases}
\end{eqnarray*}
Therefore,  we obtain (ii)  for the case  $i\geq j+2$.   Summarizing all the above,  we obtain (ii).

(iii). Suppose $i\leq j$.  By a straight-forward calculation,  we have 
\begin{eqnarray*}
&&\tilde d_i^f \circ \tilde d_j^g (v_0v_1\ldots v_n)\\
&=&\begin{cases}
\sum_{v\neq v_{j-1}, v_j}  \sum_{u\neq  v_{i-1}, v_i} 
(-1)^{i+j} f(u) g(v) v_0\ldots v_{i-1} u v_i\ldots v_{j-1} v v_j\ldots v_n,  &i\leq j-1,\\
\sum_{v\neq v_{j-1}, v_j}  \sum_{u\neq  v_{i-1}, v} 
  f(u) g(v) v_0\ldots v_{j-1} u   v v_j\ldots v_n,  &i=j. 
\end{cases}
\end{eqnarray*}
On the other hand, 
\begin{eqnarray*}
&&\tilde d_{j+1}^g \circ \tilde d_i^f (v_0v_1\ldots v_n)\\
&=&\begin{cases}
\sum_{u\neq  v_{i-1}, v_i} \sum_{v\neq v_{j-1}, v_j}  
(-1)^{i+j+1} f(u) g(v) v_0\ldots v_{i-1} u v_i\ldots v_{j-1} v v_j\ldots v_n,  &i\leq j-1,\\
-\sum_{u\neq v_{i-1}, v_i}  \sum_{v\neq  u, v_i} 
  f(u) g(v) v_0\ldots v_{i-1} u   v v_i\ldots v_n,  &i=j. 
\end{cases}
\end{eqnarray*}
Therefore,  we obtain (iii)  for $i\leq j-1$.  
\end{proof}

 \begin{remark}
The simplicial-like identities for the regular paths given in Theorem~\ref{pr-3.1.1}  are slightly different from the simplicial-like identities for the  paths given in Theorem~\ref{pr-2.3.2}.  The difference  is that   generally  in Theorem~\ref{pr-3.1.1},  the first simplicial identity does  not hold  in the case  $i=j-1$,   the second simplicial identity does  not hold  in the cases  $i=j-1$ and $i=j+1$,  and the third simplicial identity does  not hold  in the case  $i=j$. 
\end{remark}

\smallskip

\subsection{The Weighted Boundary Operators and The Weighted Co-Boundary Operators for Regular Paths}\label{ss-3.2}

 Let $f$  be a real function on $V$.  Let $n\geq 0$.   In  this subsection,   we  investigate  the  $f$-weighted boundary maps and the $f$-weighted co-boundary maps  for regular paths on $V$.    
 
 \smallskip

 \begin{definition}\label{def-3.15}
  For any $v\in V$,  we  use the notation  
 \begin{eqnarray*}
 \tilde\partial^f:= \sum_{i=0}^n \tilde\partial_i^f. 
  \end{eqnarray*}
 Then  we  have a graded linear map  
 \begin{eqnarray*}
 \tilde \partial^f: ~~~\mathcal{R}_n(V)\longrightarrow\mathcal{R}_{n-1}(V), ~~~~~~  n\geq 0. 
 \end{eqnarray*}
 We call $\tilde \partial^f$   the {\it $f$-weighted boundary map for regular paths}.  
 \end{definition}

  \begin{definition}\label{def-3.16}
  For any $v\in V$,  we  use the notation
 \begin{eqnarray*}
 \tilde d^f:= \sum_{i=0}^{n+1} \tilde d_i^f. 
 \end{eqnarray*}
 Then  we  have  a  graded linear map  
 \begin{eqnarray*}
 \tilde d^f: ~~~\mathcal{R}_n(V)\longrightarrow\mathcal{R}_{n+1}(V), ~~~~~~  n\geq 0. 
 \end{eqnarray*}
 We call   $\tilde d ^f$  the {\it $f$-weighted co-boundary map for regular paths}.  
 \end{definition}


\begin{definition}\label{def-3.17}
For any $v\in V$,   as a  particular case of Definition~\ref{def-3.15}  and  Definition~\ref{def-3.16},  we take $f$ to be the characteristic function $\chi_v$.   We define the {\it reduced partial derivative} with respect to $v$  as 
 \begin{eqnarray*}
 \frac{\tilde\partial}{\partial v}:=\tilde\partial^{\chi_v}=\sum_{i=0}^n \tilde\partial_i^{\chi_v}   
 \end{eqnarray*}
and define  the {\it reduced partial differentiation} with respect to $v$   as 
 \begin{eqnarray*}
 \tilde  d v:=\tilde  d ^{\chi_v}=\sum_{i=0}^n \tilde d_i^{\chi_v}.   
 \end{eqnarray*}
 \end{definition}
 
We give the explicit expressions of Definition~\ref{def-3.17}.   By Definition~\ref{def-3.17},   we  have graded linear maps 
 \begin{eqnarray*}
 \frac{\tilde\partial}{\partial v}:~~~\mathcal{R}_n(V)\longrightarrow \mathcal{R}_{n-1}(V), ~~~~~~ n\geq 0 
 \end{eqnarray*}
 and 
 \begin{eqnarray*}
\tilde dv:~~~\mathcal{R}_n(V)\longrightarrow \mathcal{R}_{n+1}(V), ~~~~~~ n\geq 0.  
 \end{eqnarray*}  
 For any regular  elementary $n$-path $v_0v_1\ldots v_n$  on $V$,  we have 
 \begin{eqnarray*}
  \frac{\tilde\partial}{\partial v}(v_0v_1\ldots v_n)=\sum_{i=0}^n(-1)^i \delta(v,v_i) \epsilon(v_{i-1}, v_{i+1})  v_0\ldots  \widehat{v_i} \ldots v_n 
 \end{eqnarray*}
 and 
 \begin{eqnarray*}
 \tilde d v (v_0v_1\ldots v_n)=\sum_{i=0}^{n+1} (-1)^i  \epsilon(v,v_{i-1}) \epsilon(v,v_i) v_0\ldots  v_{i-1} v v_i\ldots v_n. 
 \end{eqnarray*}
    Here we  abuse the notation   by  writing 
    \begin{eqnarray*}
    \epsilon(v,v_{-1})=1,~~~~~~  \epsilon (v, v_{n+1})=1. 
    \end{eqnarray*}

\begin{definition}
We  define 
\begin{enumerate}[(i). ]
\item
the {\it anti-commutator} $ \Big(\frac{\tilde\partial}{\partial v}, \frac{\tilde\partial}{\partial u}  \Big)$ of $\frac{\tilde\partial}{\partial v}$  and $\frac{\tilde\partial}{\partial u}$  by       \begin{eqnarray*}
    \Big(\frac{\tilde\partial}{\partial v}, \frac{\tilde\partial}{\partial u}  \Big): =\frac{\tilde\partial}{\partial v}\circ \frac{\tilde\partial}{\partial u} + \frac{\tilde\partial }{\partial u}\circ\frac{\tilde\partial}{\partial v};   
        \end{eqnarray*}
   \item
   the {\it anti-commutator} $(\tilde d v, \tilde d u   )$ of $\tilde d v$ and  $\tilde d u$  by          \begin{eqnarray*}
   (\tilde d v, \tilde d u   ): =\tilde d v\circ \tilde d u+\tilde d u\circ \tilde d v.   
        \end{eqnarray*} 
        \end{enumerate} 
        \end{definition}

  We  have the next lemma.  
    \begin{lemma}\label{le-3.2.1}
Let $u,v\in V$.   Let  $n\geq 0$.   Let  $v_0v_1\ldots v_n$  be a regular elementary $n$-path on $V$.   Then 
    \begin{eqnarray}\label{eq-3.2.a}
     \Big(\frac{\tilde\partial}{\partial v}, \frac{\tilde\partial}{\partial u}  \Big)(v_0v_1\ldots v_n)&=&\sum_{i=0}^n   \Big(\delta(v,v_i) \delta(u, v_{i-1})   -\delta(u,v_i)\delta(v,v_{i-1})   \Big)\nonumber\\
     &&~~~~~~\epsilon(v_{i-1}, v_{i+1}) \epsilon(v_{i-2}, v_{i+1})v_0\ldots \widehat{ v_{i-1}}\widehat{ v_i} \ldots v_n\nonumber \\
  && + \Big( \delta(u,v_i)\delta(v,v_{i+1})
 -\delta(v,v_i)\delta(u, v_{i+1})
    \Big)\nonumber\\
  &&  ~~~~~~\epsilon(v_{i-1}, v_{i+1}) \epsilon(v_{i-1}, v_{i+2}) v_0\ldots \widehat{v_i} \widehat{v_{i+1}} \ldots v_n    
     \end{eqnarray}
     and 
     \begin{eqnarray}\label{eq-3.2.b}
(\tilde d v, \tilde d u   )(v_0v_1\ldots v_n)
    &=&\sum_{i=0}^{n+1}\Big(  
     \epsilon (u,v_{i-1})   - 
    \epsilon(v,v_i)   \Big)\epsilon(u,v_i)\epsilon(v,v_{i-1})\epsilon(v,u)\nonumber\\
  &&  ~~~~~~v_0\ldots v_{i-1} vu v_i\ldots v_n\nonumber\\
   &&+ \Big(\epsilon(v,v_{i-1})-\epsilon(u,v_i)\Big)\epsilon(u,v_{i-1})\epsilon(v,v_i)\epsilon(v,u)\nonumber\\
  && ~~~~~~v_0\ldots v_{i-1} uv v_i\ldots v_n. 
   \end{eqnarray}
       \end{lemma}
    
    \begin{proof}
  Let $u,v\in V$,  $n\geq 0$,   and  $v_0v_1\ldots v_n$  be a regular elementary $n$-path on $V$.   It follows from a straight-forward calculation that   
    \begin{eqnarray}\label{eq-3.2.10}
    \frac{\tilde\partial}{\partial v}\circ \frac{\tilde\partial}{\partial u} (v_0v_1\ldots v_n)&=&
    \sum_{i=0}^n  (-1)^i  \delta(u,v_i) \epsilon (v_{i-1},v_{i+1})\Big(
    \sum_{j=0}^{i-2} (-1)^j \delta (v,v_j)  \epsilon (v_{j-1}, v_{j+1}) \nonumber\\
   && \text{~~~~~~~~~~~~~~~~~~~~~~~~~~~~~~~~~~~~~~~~~~~~}v_0\ldots \widehat{v_j} \ldots \widehat{v_i} \ldots v_n\nonumber\\
    &&+\sum_{j=i+1}^{n-1}  (-1)^j  \delta(v,v_{j+1})  \epsilon (v_j, v_{j+2})  v_0\ldots \widehat{v_i}\ldots \widehat{v_{j+1}} \ldots v_n\nonumber \\
    && + (-1)^{i-1}\delta(v,v_{i-1}) \epsilon(v_{i-2},v_{i+1}) v_0\ldots \widehat{v_{i-1}} \widehat{v_i} \ldots v_n\nonumber\\
    && +  (-1)^i \delta(v_i, v_{i+1})  \epsilon (v_{i-1}, v_{i+2})  v_0\ldots \widehat{v_i} \widehat{v_{i+1}}\ldots  v_n 
    \Big)
    \end{eqnarray}
    and 
    \begin{eqnarray}\label{eq-3.2.11}
    \frac{\tilde\partial }{\partial u}\circ\frac{\tilde\partial}{\partial v}(v_0v_1\ldots v_n)&=&
    \sum_{j=0}^n (-1)^j \delta(v,v_j) \epsilon(v_{j-1},v_{j+1}) \Big(
    \sum_{i=0}^{j-2} (-1)^i \delta(u,v_i) \epsilon(v_{i-1}, v_{i+1}) \nonumber\\
   && \text{~~~~~~~~~~~~~~~~~~~~~~~~~~~~~~~~~~~~~~~~~~~~}v_0\ldots \widehat{v_i} \ldots \widehat{v_j}\ldots v_n  \nonumber\\
    && +  \sum_{i=j+1}^{n-1}  (-1)^i \delta(u,v_{i+1})  \epsilon(v_i, v_{i+2}) v_0\ldots \widehat{v_j} \ldots \widehat{v_{i+1}} \ldots v_n  \nonumber\\
    && + (-1)^{j-1} \delta (u,v_{j-1}) \epsilon(v_{j-2},v_{j+1}) v_0\ldots \widehat{v_{j-1}}\widehat{v_j}\ldots v_n  \nonumber\\
    && + (-1)^j \delta(u, v_{j+1}) \epsilon (v_{j-1},v_{j+2})  v_0\ldots \widehat{v_j} \widehat{v_{j+1}}\ldots v_n
    \Big).  
    \end{eqnarray}
    Here we abuse the notations by writing 
    \begin{eqnarray*}
   & \delta(v,v_{-1})=\delta(v,v_{n+1}) =\delta(u,v_{-1}) =\delta(u,v_{n+1})=0, \\
   &\epsilon(v_{-1},v_2)= \epsilon(v_{n-2},v_{n+1})=1.
    \end{eqnarray*}
  Summing up (\ref{eq-3.2.10})  and (\ref{eq-3.2.11}),  it follows that   
      \begin{eqnarray*}
  \Big(\frac{\tilde\partial}{\partial v}, \frac{\tilde\partial}{\partial u}  \Big)(v_0v_1\ldots v_n)&=&
  \sum_{i=0}^n \Big( -\delta(u,v_i)\delta(v,v_{i-1}) \epsilon(v_{i-1}, v_{i+1}) \epsilon(v_{i-2}, v_{i+1}) v_0\ldots \widehat{ v_{i-1}}\widehat{ v_i} \ldots v_n \\
  && +   \delta(u,v_i)\delta(v,v_{i+1})\epsilon(v_{i-1},v_{i+1}) \epsilon(v_{i-1}, v_{i+2}) v_0\ldots \widehat{v_i} \widehat{v_{i+1}} \ldots v_n
   \Big)\\
   &&- \sum_{j=0}^n \Big( -\delta(v,v_j) \delta(u, v_{j-1}) \epsilon (v_{j-1}, v_{j+1}) \epsilon (v_{j-2}, v_{j+1}) v_0\ldots \widehat {v_{j-1}} \widehat{ v_j} \ldots  v_n  \\
   && +  \delta(v,v_j)\delta(u, v_{j+1})\epsilon(v_{j-1}, v_{j+1}) \epsilon(v_{j-1}, v_{j+2}) v_0\ldots \widehat{v_{j}} \widehat{v_{j+1}} \ldots v_n
   \Big)\\
   &=&\sum_{i=0}^n   \Big(\delta(v,v_i) \delta(u, v_{i-1})   -\delta(u,v_i)\delta(v,v_{i-1})   \Big)\epsilon(v_{i-1}, v_{i+1}) \epsilon(v_{i-2}, v_{i+1})\nonumber\\
   && \text{~~~~~~}v_0\ldots \widehat{ v_{i-1}}\widehat{ v_i} \ldots v_n\\
  && + \Big( \delta(u,v_i)\delta(v,v_{i+1})
 -\delta(v,v_i)\delta(u, v_{i+1})
    \Big)\epsilon(v_{i-1}, v_{i+1}) \epsilon(v_{i-1}, v_{i+2}) \nonumber\\
   && \text{~~~}v_0\ldots \widehat{v_i} \widehat{v_{i+1}} \ldots v_n. 
    \end{eqnarray*}  
    We obtain (\ref{eq-3.2.a}).

     On the other hand,  it also follows from a straight-forward calculation that 
         \begin{eqnarray}\label{eq-3.2.12}
    \tilde d v\circ \tilde d u (v_0v_1\ldots v_n)&=&\sum_{i=0}^{n+1} (-1)^i \epsilon (u,v_{i-1}) 
    \epsilon(u,v_i) \Big(
    \sum_{j=0}^{i-1} (-1)^j \epsilon (v,v_{j-1}) \epsilon (v,v_j) \nonumber\\
   && \text{~~~~~~~~~~~~~~~~~~~~~~~~~~~~~~~~~~~~~~~~}v_0\ldots v_{j-1} v v_j\ldots v_{i-1} u v_i\ldots v_n\nonumber\\
   && + \sum_{j=i+2}^{n+2} (-1)^j \epsilon (v,v_{j-2})  \epsilon (v,v_{j-1}) v_0\ldots v_{i-1}u v_i\ldots v_{j-2} v v_{j-1}\ldots  v_n \nonumber\\
  && +  (-1)^i \epsilon (v,v_{i-1})  \epsilon (v,u) v_0\ldots v_{i-1} v u v_i\ldots v_n \nonumber
  \\
  && + (-1)^{i+1} \epsilon (v,u) \epsilon(v,v_i) v_0\ldots v_{i-1}  u v v_i\ldots v_n
    \Big)
    \end{eqnarray}
      and 
      \begin{eqnarray}\label{eq-3.2.13}
    \tilde d u\circ \tilde d v (v_0v_1\ldots v_n)&=&\sum_{j=0}^{n+1} (-1)^j \epsilon (v,v_{j-1}) 
    \epsilon(v,v_j) \Big(
    \sum_{i=0}^{j-1} (-1)^i \epsilon (u,v_{i-1}) \epsilon (u,v_i)\nonumber\\
   && \text{~~~~~~~~~~~~~~~~~~~~~~~~~~~~~~~~~~~~~~~~} v_0\ldots v_{i-1} u v_i\ldots v_{j-1} v v_j\ldots v_n\nonumber \\
   && + \sum_{i=j+2}^{n+2} (-1)^i \epsilon (u,v_{i-2})  \epsilon (u,v_{i-1}) v_0\ldots v_{j-1}v v_j\ldots v_{i-2} u v_{i-1}\ldots  v_n \nonumber \\
  && +  (-1)^j \epsilon (u,v_{j-1})  \epsilon (v,u) v_0\ldots v_{j-1} u v v_j\ldots v_n 
  \nonumber \\
  && + (-1)^{j+1} \epsilon (v,u) \epsilon(u,v_j) v_0\ldots v_{j-1}  v u v_j\ldots v_n 
    \Big).  
    \end{eqnarray}
 Summing up  (\ref{eq-3.2.12})  and (\ref{eq-3.2.13}),   it follows that  
 \begin{eqnarray*}
(\tilde d v, \tilde d u   )(v_0v_1\ldots v_n)
&=&\sum_{i=0}^{n+1}\Big(\epsilon (u,v_{i-1}) 
    \epsilon(u,v_i) \epsilon (v,v_{i-1})  \epsilon (v,u) v_0\ldots v_{i-1} v u v_i\ldots v_n \\
    &&-\epsilon (u,v_{i-1}) 
    \epsilon(u,v_i)\epsilon (v,u) \epsilon(v,v_i) v_0\ldots v_{i-1}  u v v_i\ldots v_n \Big)\\
    && +\sum_{j=0}^{n+1}\Big(\epsilon (v,v_{j-1}) 
    \epsilon(v,v_j)\epsilon (u,v_{j-1})  \epsilon (v,u) v_0\ldots v_{j-1} u v v_j\ldots v_n 
  \\  
   && - \epsilon (v,v_{j-1}) 
    \epsilon(v,v_j) \epsilon (v,u) \epsilon(u,v_j) v_0\ldots v_{j-1}  v u v_j\ldots v_n \Big)\\
    &=&\sum_{i=0}^{n+1}\Big(  
     \epsilon (u,v_{i-1})   - 
    \epsilon(v,v_i)   \Big)\epsilon(u,v_i)\epsilon(v,v_{i-1})\epsilon(v,u)v_0\ldots v_{i-1} vu v_i\ldots v_n\\
   &&+ \Big(\epsilon(v,v_{i-1})-\epsilon(u,v_i)\Big)\epsilon(u,v_{i-1})\epsilon(v,v_i)\epsilon(v,u)v_0\ldots v_{i-1} uv v_i\ldots v_n. 
   \end{eqnarray*}
   We  obtain (\ref{eq-3.2.b}).  
\end{proof}
 
 Let $f$ and $g$  be two real functions on $V$  We  note that 
 \begin{eqnarray*}
 \tilde\partial^f= \sum_{v\in V}  f(v) \frac{\tilde\partial}{\partial v}, ~~~~~~ 
 \tilde d^f=\sum_{v\in V}  f(v) \tilde d v
 \end{eqnarray*}
 and the same identities hold for $g$ as well.  Moreover,  if we write  the anti-commutators as 
 \begin{eqnarray*}
 (\tilde \partial^f, \tilde \partial^g)=\tilde \partial^f\circ \tilde \partial^g+ \tilde \partial^g \circ \tilde \partial^f, ~~~~~~  (\tilde d^f,\tilde d^g)=\tilde d^f\circ \tilde d^g + \tilde d^g\circ \tilde d^f,
 \end{eqnarray*}
 then  we  have 
 \begin{eqnarray}\label{eq-3.2.aaa}
 (\tilde \partial^f, \tilde \partial^g)=\sum_{v,u\in V}f(v)g(u)    \Big(\frac{\tilde\partial}{\partial v}, \frac{\tilde\partial}{\partial u}  \Big)  
 \end{eqnarray}
 and 
 \begin{eqnarray}\label{eq-3.2.bbb}
 (\tilde d^f,\tilde d^g)=\sum_{v,u\in V}f(v)g(u)  (\tilde d v,\tilde d u). 
 \end{eqnarray}
 The next proposition follows from Lemma~\ref{le-3.2.1}. 
 
 \begin{proposition}\label{le-3.2.88}
 Let $f$ and $g$ be real functions on $V$.   Let  $n\geq 0$  and let  $v_0v_1\ldots v_n$  be a regular elementary $n$-path on $V$.   Then 
    \begin{eqnarray}\label{eq-3.2.aa}
     (\tilde\partial^f, \tilde\partial^g)(v_0v_1\ldots v_n)&=&\sum_{v,u\in V} f(v) g(u) \Big[\sum_{i=0}^n   \Big(\delta(v,v_i) \delta(u, v_{i-1})   -\delta(u,v_i)\delta(v,v_{i-1})   \Big)\nonumber\\
     &&~~~~~~\epsilon(v_{i-1}, v_{i+1}) \epsilon(v_{i-2}, v_{i+1})v_0\ldots \widehat{ v_{i-1}}\widehat{ v_i} \ldots v_n\nonumber \\
  && + \Big( \delta(u,v_i)\delta(v,v_{i+1})
 -\delta(v,v_i)\delta(u, v_{i+1})
    \Big)\nonumber\\
  &&  ~~~~~~\epsilon(v_{i-1}, v_{i+1}) \epsilon(v_{i-1}, v_{i+2}) v_0\ldots \widehat{v_i} \widehat{v_{i+1}} \ldots v_n \Big]   
     \end{eqnarray}
     and 
     \begin{eqnarray}\label{eq-3.2.bb}
(\tilde d ^f, \tilde d ^g  )(v_0v_1\ldots v_n)
    &=&\sum_{v,u\in V} f(v)g(u) \Big[\sum_{i=0}^{n+1}\Big(  
     \epsilon (u,v_{i-1})   - 
    \epsilon(v,v_i)   \Big)\epsilon(u,v_i)\epsilon(v,v_{i-1})\epsilon(v,u)\nonumber\\
  &&  ~~~~~~v_0\ldots v_{i-1} vu v_i\ldots v_n\nonumber\\
   &&+ \Big(\epsilon(v,v_{i-1})-\epsilon(u,v_i)\Big)\epsilon(u,v_{i-1})\epsilon(v,v_i)\epsilon(v,u)\nonumber\\
  && ~~~~~~v_0\ldots v_{i-1} uv v_i\ldots v_n \Big]. 
   \end{eqnarray}
 \end{proposition}
 
 \begin{proof}
 The expression (\ref{eq-3.2.aa}) follows from (\ref{eq-3.2.a}) and (\ref{eq-3.2.aaa}).  And the expression (\ref{eq-3.2.bb}) follows from (\ref{eq-3.2.b})  and (\ref{eq-3.2.bbb}).   
 \end{proof}
 
 \begin{remark}
 By Proposition~\ref{le-3.2.88},  the operators $\tilde\partial^f$ and $\tilde\partial^g$  as well  as the operators $\tilde d ^f$ and $\tilde d ^g$ on the regular paths  are not anti-commutative, in general.  The anti-commutative property  holds in Proposition~\ref{pr-2.5.5}  and does not hold  in Proposition~\ref{le-3.2.88}. 
 \end{remark}
 
 \bigskip

\section*{Acknowledgement} {The  author would like to express his  deep
gratitude to the referee for the careful reading of the manuscript.}

 \bigskip
 
Shiquan Ren 

Address:   School of Mathematics and Statistics,  Henan University,  Kaifeng  475004,  China. 

E-mail:  srenmath@126.com

\end{document}